\documentclass[a4paper,11pt]{amsart}

\usepackage{amsmath}
\usepackage{amssymb}
\usepackage{amsthm}
\usepackage{verbatim}
\usepackage[all]{xy}
\usepackage{enumerate}
\usepackage{setspace}
\usepackage{color}

\usepackage[all]{xy}

\usepackage{color}

\newtheorem{thm}{Theorem}[section]
\newtheorem{prop}[thm]{Proposition}
\newtheorem{cor}[thm]{Corollary}
\newtheorem{lem}[thm]{Lemma}

\theoremstyle{definition}

\newtheorem{dfn}[thm]{Definition}

\newtheorem{rmk}[thm]{Remark}

\numberwithin{equation}{section}
\newcommand{\cM}{\mathcal{M}}
\newcommand{\cH}{\mathcal{H}}

\newcommand{\cB}{\mathcal{B}}

\newcommand{\nphi}{\mathfrak{n}_{\varphi}}
\newcommand{\npsi}{\mathfrak{n}_{\psi}}
\newcommand{\mphi}{\mathfrak{m}_{\varphi}}

\newcommand{\cN}{\mathcal{N}}

\newcommand{\bA}{\mathsf{A}}

\newcommand{\cMext}{\mathcal{M}^+_{{\rm ext}}}

\newcommand{\cNext}{\mathcal{N}^+_{{\rm ext}}}
\newcommand{\Dom}{\textrm{Dom}}

\newcommand{\Tpsiop}{  \mathcal{T}_{ \psi^{{\rm op}} } }
\newcommand{\npsiop}{  \mathfrak{n}_{ \psi^{{\rm op}}  }}

\newcommand{\psiop}{   \psi^{{\rm op}}   }

\title[Weight theory for ultraproducts]{Weight theory for ultraproducts}

\date{\noindent \today.
 }

\author{Martijn Caspers}
 \address{M. Caspers, Universiteit Utrecht,
Budapestlaan 6, 3584 CD Utrecht,
The Netherlands}
 \email{m.p.t.caspers@uu.nl}

\begin{document}


\begin{abstract}
For a family of von Neumann algebras $\cM_j$ equipped with normal weights $\varphi_j$ we define the ultraproduct weight $(\varphi_j)_\omega$ on the Groh--Raynaud ultrapower $\prod_{j, \omega} \cM_j$. We prove results about Tomita-Takesaki modular theory and consider ultraproducts of spatial derivatives. This extends results by Ando--Haagerup and Raynaud for the state case. We give some applications to noncommutative $L^p$-spaces  and indicate how ultraproducts of weights appear naturally in transference results for Schur and Fourier multipliers. Using ideas from complex interpolation with respect to ultraproduct weights, we give a new proof of a theorem by Raynaud which shows that $\prod_{j, \omega} L^p(\cM_j) \simeq L^p(\prod_{j, \omega} \cM_j )$. We complement the paper by showing that spatial derivatives take a  natural form in terms of noncommutative $L^p$-spaces.
\end{abstract}

\maketitle

\section{Introduction}

Ultraproducts form a very useful tool in the study of Banach spaces. This was first observed in \cite{DucKr} after which many applications were found. We mention a couple here. Firstly in von Neumann algebra theory ultraproducts play a very prominent role. This was for example the case in the classification of hyperfinite factors \cite{ConnesClass}, \cite{HaagerupClass}. Also Connes' famous -- still unresolved -- embedding problem (equivalently the QWEP conjecture) is formulated in terms of ultraproducts. Recall that it states that every separable II$_1$-factor embeds in an (Ocneanu) ultraproduct of the hyperfinite II$_1$-factor. There are (further) strong connections of ultraproducts with free probabiltiy theory, see e.g. \cite{JungeInventiones}, \cite{PopaJFA} or \cite{HoudayerIsono} or model theory/descriptive set theory, see e.g. \cite{FSH}, \cite{Boutonnet}.

A second application can be found in noncommutative harmonic analysis, i.e. harmonic analysis for nonabelian groups, harmonic analysis on group von Neumann algebras or vector valued semi-commutative analysis. It is often very useful to asymptotically intertwine Fourier multipliers to obtain sharp bounds for them or to discretize them. This asymptotic intertwining technique or ``transference method'' was applied in many places, see for instance \cite{NeuwirthRicard}, \cite{CaspersSalle}, \cite{CPPR}, \cite{Ricard}, \cite{Gonzalez}. We may view such an asymptotic intertwining property as a (non-asymptotic) intertwining property between ultraproduct maps (we omit further details here but refer to these references and further discussions below).

 Whereas the definition of ultraproducts of Banach spaces is rather elementary it is often quite non-obvious what properties the ultraproduct Banach space has. For example, if one takes an ultraproduct of $L^p$-spaces, it is a non-trivial result that the resulting space is again an $L^p$-space.  The latter fact is due to Raynaud \cite{Raynaud}. In \cite{Raynaud} Raynaud also proved that the ultraproduct of the Tomita--Takesaki modular automorphism group of states localizes to the modular automorphism group of the ultraproduct state, i.e. the bounded case of Theorem A below. In addition Raynaud proves a similar result for the Connes cocycle derivative.

More recently, in \cite{AndoHaagerup} Ando and Haagerup gave a   complete outline of different ultraproducts of von Neumann algebras and their relations. In particular they considered the Groh-Raynaud ultraproduct and the Ocneanu ultraproduct and showed that the Ocneanu ultraproduct is isomorphic to  a corner algebra of the Groh-Raynaud ultraproduct. Ando and Haagerup proved many things about ultraproducts. For example they extended Raynaud's results on Tomita--Takesaki theory to ultraproducts of possibly  different von Neumann algebras with a new proof. They were also able to answer questions on asymptotic central sequences, type and factoriality of ultraproducts.

\vspace{0.3cm}

 In this paper we shall only be dealing with the Groh-Raynaud ultraproduct whose name originates from \cite{Groh} and \cite{Raynaud}. Recall that it is defined as follows. Let $\omega$ be a nonprincipal ultrafilter on $\mathbb{N}$ and for every $j \in \mathbb{N}$ let $\cM_j$ be a  von Neumann algebra. Assume that $\cM_j$ is represented on the standard Hilbert space $\cH_j$. Let $\cH = \prod_{j, \omega} \cH_j$ be the ultraproduct Hilbert space. A bounded sequence $(x_j)_\omega$ with $x_j \in \cM_j$ acts on $\cH$ by $(x_j)_{\omega} (\xi_j)_\omega = (x_j \xi_j)_\omega$. The closure in the strong topology of such $(x_j)_\omega$ in $\cB(\cH)$ is called the Groh-Raynaud ultraproduct. We denote this ultraproduct by $\cM$.  In \cite{Raynaud} (see also \cite{Groh}) it was proved that
$\cM_\ast \simeq \prod_{j, \omega} (\cM_j)_\ast$  by extending the pairing $\langle  (\rho_j)_\omega, (x_j)_\omega \rangle = \lim_{j, \omega} \rho_j(x_j)$.

The aim of this paper is the study of ultraproduct weights. Whereas one can easily define ultraproducts of (normal) functionals it is {\it a priori} not clear how ultraproducts of normal weights should be defined. The problem lies in the fact that there is no bound on a sequence $(\varphi_j)_\omega$ and therefore a naive definition $\langle (\varphi_j)_\omega, (x_j)_\omega \rangle = \lim_{j, \omega} \varphi_j(x_j)$ does not make sense. Indeed it could very well be that the right hand side of this expression is non-zero whereas $(x_j)_\omega$ is (the equivalence class of) the 0 element.

In \cite{AndoHaagerup} ultraproducts of weights were discussed briefly for the first time. The approach is   defined  in the case that $\varphi_j$ is a fixed weight on a fixed von Neumann algebra (i.e. the weight and von Neumann algebra do not depend on $j$). In this case one sets $(\varphi)_\omega = \varphi \circ \mathcal{E}$ where $\mathcal{E}$ is the conditional expectation of the ultraproduct on the von Neumann subalgebra of fixed sequences.  This gives a {\it different} notion of ultraproduct weights than Definition \ref{Dfn=Intro}, see Remark \ref{Rmk=AndoHaagerupWeight}. The weights from \cite{AndoHaagerup} are sometimes insufficient for applications. For example one cannot prove Corollary \ref{Cor=Lp} because there is very few control over the supports of the ultraproduct weights from \cite{AndoHaagerup}. Moreover, typical weights appearing in the literature {\it do} depend on the index.  For example the weights occuring in \cite[Section 5]{CaspersSalle} are given by ${\rm Tr}(P_F \cdot  P_F)$ on $\cB(L^2(G))$ where $F$ are F\o{}lner sets on a locally compact group $G$ (the F\o{}lner sets give the net structure).

\vspace{0.3cm}

In this paper we introduce ultraproducts of weights in the following way.

 \begin{dfn}\label{Dfn=Intro} Let $\varphi_j$ be a normal (not necessarily faithful or semi-finite) weight on $\cM_j$. Then we define the ultraproduct weight $\varphi = (\varphi_j)_\omega$ as
 \[
 (\varphi_j)_\omega :=  \sup \: \{  \rho = (\rho_j)_\omega \mid \rho_j \in (\cM_j)_\ast^+, \rho_j \leq \varphi_j, \Vert \rho_j \Vert \textrm{ bounded in } j \}.
 \]
 \end{dfn}

 \vspace{0.3cm}

Note that the definition  agrees with the ultraproduct of bounded normal functionals. We now summarize the main results of this paper. Firstly we prove that the modular theory of an ultraproduct localizes in the sense of the following theorem.

\vspace{0.3cm}

\noindent {\bf Theorem A.} Let $\cM_j$ be von Neumann algebras with normal, semi-finite, faithful weights $\varphi_j$. Let $\cM = \prod_{j, \omega} \cM_j$ and let $\varphi = (\varphi_j)_\omega$ be the ultraproduct weight on $\cM$. Assume that $\varphi$ is semi-finite. Let $p_\varphi$ be the support projection of $\varphi$. Then,
\[
\sigma^{\varphi}_t(p_\varphi (x_j)_\omega p_\varphi) =  p_\varphi (\sigma^{\varphi_j}_t(x_j))_\omega p_\varphi.
\]

\vspace{0.3cm}

We also prove the analogue of Theorem A for cocycle derivatives, see Theorem \ref{Thm=CocycleDerivative}. Theorem A  generalizes the result in \cite{Raynaud} and \cite{AndoHaagerup} where it was proved   for the case that  $\varphi_j$'s  are   bounded uniformly in $j$. Note that \cite{Raynaud} and \cite{AndoHaagerup} give different proofs. In this paper we give a third proof through the theory of spatial derivatives. In fact we prove the following, which is (opposed to Theorem A) new even in the case that the $\varphi_j$'s are  bounded functionals.

\vspace{0.3cm}

\noindent {\bf Theorem B.} Let $\cM_j$ be von Neumann algebras with normal, semi-finite, faithful weights $\varphi_j$. Let $\cM = \prod_{j, \omega} \cM_j$ and let $\varphi = (\varphi_j)_\omega$ be the ultraproduct weight on $\cM$ with support projection $p_\varphi$. Assume that $\varphi$ is semi-finite.  Let $J_j$ be the modular conjugation acting on the standard Hilbert space $\cH_j$ of   $\cM_j$ and let $J_\omega = (J_j)_\omega$ which acts on $\cH = \prod_{j, \omega} \cH_j$. Suppose that $\psi_j$ are normal, semi-finite, faithful weights on the commutant $\cM_j'$ with ultraproduct $\psi = (\psi_j)_\omega$.  Assume that $\psi$ is semi-finite.  Let $r_\psi$ be the support of $\psi$ and set $p_\psi = J_\omega r_\psi J_\omega, q_\psi = p_\psi r_\psi$. Assume that $p_\varphi \leq p_\psi$. Then $q_\psi(\cH)$ is an invariant subspace of $\left( \frac{d \varphi_j}{ d \psi_j} \right)_\omega$ and,
\begin{equation}\label{Eqn=ThmB}
   \frac{d \varphi}{ d \psi}  = q_\psi \left( \frac{d \varphi_j}{ d \psi_j} \right)_\omega q_\psi.
\end{equation}

\vspace{0.3cm}

 The statement of Theorem B uses ultraproducts of positive self-adjoint operators which we define in this paper and takes values in the extended positive cone. So \eqref{Eqn=ThmB} needs to be interpreted in the extended positive cone of $\cM$. Along the way of proving Theorem B we also discuss whether or not every normal weight on $\cM$ can be written as an ultraproduct of weights on $\cM_j$. For states this is true, but for weights it no longer holds, see Proposition \ref{Prop=Decomposition}. In fact  it seems that normal, semi-finite and  faithful  weights on an ultraproduct von Neumann algebras rarely appear as an ultraproduct.

Using Theorem B we can show that ultraproduct von Neumann algebras admit locally a nice interpolation structure. As a consequence we reprove Raynaud's theorem \cite[Theorem 3.1]{Raynaud}.

\vspace{0.3cm}

\noindent {\bf Theorem C.} We have,
\begin{equation}\label{Eqn=ThmC}
L^p(\prod_{j, \omega}  \cM_j ) \simeq \prod_{j, \omega} L^p(\cM_j).
\end{equation}

 \vspace{0.3cm}

The crucial novelty lies in {\it how}  Theorem C was proved, see Theorem \ref{Thm=Lp} and Corollary \ref{Cor=Lp} (we found it more suitable to surpress these slightly more technical statements in Theorem C making them more explicit in Section \ref{Sect=Lp}).  We show that the isomorphism of \eqref{Eqn=ThmC} is canonical and allows us to see $\prod_{j, \omega} L^p(\cM_j)$ locally as complex interpolation spaces in a natural way through this isomorphism.   In many situations this complex interpolation is useful. For example   in \cite{CaspersSalle} isometric embeddings of $L^p$-spaces were considered. It follows from our result that these isometries may be obtained through interpolation as well. Similar situations/ideas occur in \cite{CPPR} and \cite{Gonzalez}.

In the last part of this paper, in Appendix \ref{Sect=AppendixA}, we show that spatial derivatives take a natural form in terms of noncommutative $L^p$-spaces. These results are probably known amongst experts but have not appeared in the literature so far.

\vspace{0.3cm}

\noindent {\it Structure.} In Section \ref{Sect=Prelim} we recall preliminaries on ultraproducts and Tomita-Takesaki theory. Section \ref{Sect=LpPrelim} contains preliminaries on noncommutative $L^p$-spaces associated with an arbitrary von Neumann algebra. In Section \ref{Sect=Ultra} we introduce and study ultraproducts of weights. In particular we prove Theorem A and Theorem B.   Section \ref{Sect=Lp} is devoted to a proof of Theorem C. The Appendix \ref{Sect=AppendixA} contains explicit forms of spatial derivatives on the standard form. We use these results in the main part of this paper.

\vspace{0.3cm}

\noindent {\it General notation.} We use $[x]$ for the closure of an operator $x$ and $x \cdot y = [xy]$. 

\section{Preliminaries}\label{Sect=Prelim}
For standard results on von Neumann algebras we refer to \cite{TakI} and \cite{TakII} and adopt most of its notation. Von Neumann algebras will typically be denoted by $\cM$ and $\cN$ and Hilbert spaces by $\cH$. $\cN$ is usually an arbitrary von Neumann algebra whereas $\cM$ shall be an ultraproduct.  We use $\cN_\ast$ for the predual and it induces the $\sigma$-weak topology on $\cN$. We freely use the various other topologies on $\cN$ summarized in \cite[Theorem II.2.6]{TakI}. $\cN^+$ is the cone of positive (bounded) operators in $\cN$ and $\cN_\ast^+$ the space of positive normal functionals. For $x \in \cN$ and $\rho \in \cN_\ast$ we write $\rho x$ for the functional $(\rho x)(y)  = \rho(xy)$ and similarly $(x \rho)(y) = \rho(yx)$. For $\xi \in \cH$ we write $\rho_{\xi, \xi}$ for the inner product functional $\rho_{\xi, \xi}(x) = \langle x \xi, \xi \rangle$.

\subsection{Weight theory} Let $\cN$ be a von Neumann algebra. A {\it weight} on $\cN$ is a map $\varphi: \cN^+ \rightarrow [0, \infty]$ that satisfies:
\begin{enumerate}
\item $\varphi(\lambda x) = \lambda \varphi(x)$ for $x \in \cN^+$ and  $\lambda \geq 0$;
 \item $\varphi(x+y) = \varphi(x) + \varphi(y)$ for $x,y \in \cN^+$.
 \end{enumerate}
 Alternatively one can say that $\varphi$ preserves convex combinations on the cone $\cN^+$. We set,
 \[
 \nphi = \{ x \in \cN \mid \varphi(x^\ast x) < \infty \}.
 \]
 From the inequality $x^\ast y^\ast y x \leq \Vert y \Vert^2 x^\ast x$ it follows that $\varphi(x^\ast y^\ast y x) \leq \Vert y\Vert^2 \varphi(x^\ast x)$ and therefore $\nphi$ is a left ideal. We set $\mphi$ to be the span of $\nphi^\ast \nphi$. Also set $\mphi^+ = \mphi \cap \cN^+$.  Then $\mphi$ is the linear span of $\mphi^+$. Recall the following definitions.
 \begin{enumerate}
 \item The weight $\varphi$ is called {\it faithful} if $\varphi(x) = 0, x \in \cN^+$ implies that $x = 0$;
 \item The weight $\varphi$ is called {\it semi-finite} if $\nphi$ is $\sigma$-weakly dense in $\cN$. This is equivalent to $\sigma$-weak density of $\mphi$ in $\cN$ or $\sigma$-weak density if $\mphi^+$ in $\cN^+$;
 \item The weight $\varphi$ is called {\it normal} if $\varphi(\sup_k x_k) = \sup_k \varphi(x_k)$ for every bounded increasing net $\{x_k\}_{k \in K}$ in $\cN^+$.
\end{enumerate}

Every weight in this paper shall be normal and from this point we assume that so is $\varphi$. Let $e \in \cN$ be the projection such that $\cN e$ equals the $\sigma$-strong closure of $\nphi$. Let $f \in \cN$ be the projection such that $\cN f$ equals $\{ x \in \cN \mid \varphi(x^\ast x ) = 0 \}$. $\varphi$ is semi-finite on $e \cN e$ and faithful on $f \cN f$. The projection $e - f$ is called the {\it support} of $\varphi$ and we denote it by $p_\varphi$. Note that in this paper we also consider supports of weights  $\psi$ on the commutant $\cN'$ which shall be denoted by $r_\psi$ to distinguish. If $\cN$ is represented on the standard Hilbert space we set the opposite support $p_\psi = J r_\psi J$ where $J$ is the modular conjugation of the standard form defined below.

  Now assume that $\varphi$ is normal, semi-finite and faithful.  We have a non-degenerate inner product,
\[
\nphi \times \nphi \rightarrow \mathbb{C}: (x,y) \mapsto \varphi(y^\ast x),
\]
from which we may complete $\nphi$ to a Hilbert space $\cH_\varphi$. We write $\Lambda_\varphi: \nphi \rightarrow \cH_\varphi$ for the identification map. $\nphi$ is a $\sigma$-weak/norm closed map.  We have a normal representation of $\cN$ on $\cH_\varphi$ given by $\pi_\varphi(x) \Lambda_\varphi(y) = \Lambda_\varphi(xy)$ where $x \in \cN$ and $y \in \nphi$. We may identify $\cN$ with its image $\pi_\varphi(\cN)$ and shall omit $\pi_\varphi$ in the notation. We let,
\[
S_0: \: (\subseteq \mathcal{H}_\varphi) \rightarrow \cH_\varphi: \Lambda_{\varphi}(x) \mapsto \Lambda_\varphi(x^\ast), \qquad x \in \nphi \cap \nphi^\ast.
\]
$S_0$ is preclosed and we set is closure to be $S$. $S$ has polar decomposition $S = J \Delta^{\frac{1}{2}}$. Here the anti-linear isometry $J$ is called the {\it modular conjugation} and the positive operator $\Delta$ is called the {\it modular operator}. We set $\sigma^\varphi_t(x) = \Delta^{it} x \Delta^{-it}, t \in \mathbb{R}, x \in \cN$ which is the {\it modular automorphism group} of $\varphi$. Tomita-Takesaki theory shows that $\sigma^\varphi$ is a strongly continuous 1--parameter group of automorphisms of $\cN$ (i.e. the non-trivial part is that it preserves $\cN$).

We need a couple of standard tools from modular theory. Firstly we set,
\[
\mathcal{T}_\varphi = \{ x \in \cN \mid x \textrm{ is analytic with respect to } \sigma^\varphi \textrm{ and } \forall z \in \mathbb{C}: \sigma^\varphi_z(x) \in \nphi \cap \nphi^\ast \}.
\]
If we view $\mathcal{T}_\varphi$ as  a subset of $\cH_\varphi$ then
$\mathcal{T}_\varphi$ is a {\it Tomita algebra} as explained in \cite{TakII}.  For $a,b \in \mathcal{T}_\varphi$ there exists a normal state $\varphi_{ab}$ that is defined by
\[
\varphi_{ab}(x) =  \varphi( bx \sigma_{-i}^\varphi(a)).
\]
 The set of such states is  dense in $\cN_\ast$. If $\varphi$ is a state then $\varphi_{ab} = \varphi ab$.
 The following lemma is a standard approximation argument. Recall that on bounded sets the strong and $\sigma$-strong topology agree so that the lemma can in fact be derived from (the proof of) \cite[Lemma 9]{TerpII}.
\begin{lem}\label{Lem=Approx}
There exists a net $a_j \in \mathcal{T}_\varphi$ such that for every $z \in \mathbb{C}$ the net $\Vert \sigma^\varphi_z(a_j) \Vert$ is bounded and such that  $a_j \rightarrow 1$ and $\sigma^{\varphi}_{i/2}(a_j) \rightarrow 1$ in the $\sigma$-strong topology.
\end{lem}

\subsection{Standard forms} We first recall the definition of the standard form as it was reduced to in  \cite[Lemma 3.19]{AndoHaagerup}. See also \cite{HaagerupScan}.

 \begin{dfn} \label{Dfn=Standard}
Consider a  4-tuple $(\cN, \cH, J, \mathcal{P})$ consisting of a von Neumann algebra $\cN$ that is represented on a Hilbert space $\cH$, an anti-linear isometry $J$ on $\cH$ with $J^2 = 1$ and $\mathcal{P} \subseteq \cH$ a closed convex cone which is self-dual, i.e. $\mathcal{P} = \mathcal{P}^0$ where
\[
\mathcal{P}^{0} = \{ \xi \in \cH \mid \langle \xi, \eta \rangle \geq 0, \eta \in \mathcal{P} \}.
\]
Then $(\cN, \cH, J, \mathcal{P})$ is called a {\it standard form} if $J \cN J = \cN'$, $J \xi = \xi, \xi \in \mathcal{P}$, $x JxJ(\mathcal{P}) \subseteq \mathcal{P}, x \in \cN$.
 \end{dfn}

 Standard forms are unique in the sense that if $(\cN', \cH', J', \mathcal{P}')$ is another standard form for which $\cN$ and $\cN'$ are isomorphic, then there exists a   unitary $U: \cH \rightarrow \cH'$ such that $\cN = U^\ast \cN' U, J = U^\ast J' U$ and $U \mathcal{P} = \mathcal{P}'$. Moreover, this unitary is unique. For $\rho \in \cN_\ast$ there exists a unique vector $\xi \in \mathcal{P}$ such that $\rho = \rho_{\xi, \xi}$. We will denote this vector with $D_\rho^{\frac{1}{2}}$, see below.


\subsection{The extended positive cone}\label{Sect=ExtendedPos}
Let $\cN$ be a von Neumann algebra. The extended positive cone $\cNext$ is defined as the space of all lower semi-continuous mappings $\cN_\ast^+ \rightarrow [0, \infty]$ that preserve convex combinations. Note that $\cNext$ is itself a cone, meaning that it is closed under taking convex combinations and multiplication by scalars $\geq 0$. Each $x \in \cNext$ admits a spectral resolution
\begin{equation}\label{Eqn=Integral}
\langle x, \rho \rangle_{\cNext, \cN_\ast^+ } = \int_{0}^\infty \lambda \: d\rho(e_\lambda) + \rho(p) \cdot \infty,
\end{equation}
where $e_\lambda, \lambda \in [0, \infty)$ is an increasing net of projections and $p = 1 - \sup_\lambda e_\lambda$. We take the convention $0 \cdot \infty = 0$. We may naturally view $\cN^+$ inside $\cNext$ as the set of all elements for which the projection $p$ on the infinite part is 0 and $e_\lambda$ is the spectral resolution of a bounded operator (that is, there exists $\lambda$ such that $e_\lambda = 1$). For $x \in \cNext$ and $q \in \cN$ a projection we say that $x$ and $q$ commute if all $e_\lambda, \lambda \geq 0$ and $p$ in \eqref{Eqn=Integral} commute with $q$. We shall also write $P_\infty(x)$ for $p$ and $P_{< \infty}(x)$ for $1- p$.

\subsection{Spatial derivatives}

Spatial derivatives for von Neumann algebras have been introduced by Connes in \cite{Connes}. A comprehensive summary may also be found in \cite{TerpII}.
Let $\cN$ be a von Neumann algebra acting on a Hilbert space $\cH$. We emphasize that in general $\cH$ is not necessarily the  standard form Hilbert space, but in this paper this is always the case.   We fix a normal, semi-finite, faithful weight $\psi$ on the commutant $\cN'$. Let $(\cH_\psi, \Lambda_\psi, \pi_\psi)$ be a GNS-construction for $\psi$. For $\xi \in \cH$ we define $R^\psi(\xi)$ as the densely defined operator given by the closure of the (preclosed) mapping:
\begin{equation}\label{Eqn=ROperator}
 \Lambda_\psi(x) \mapsto x \xi, \qquad x \in \mathfrak{n}_{\psi}.
\end{equation}
 Let $D(\cH, \psi)$ denote the vectors for which the operator $R^\psi(\xi)$ is bounded.  The vectors in $D(\cH, \psi)$ are  called the {\it $\psi$-bounded vectors}. It is easy to check that $R^\psi(x \xi) = x R^\psi(\xi), x \in \cN$ and in particular that $D(\cH, \psi)$ is invariant under $\cN$. Also $R^\psi( \xi) \pi_\psi(y) = y R^\psi(\xi), y \in \cN'$ so that it follows that for $\xi \in D(\cH, \psi)$ the operator $R^\psi(\xi) R^\psi(\xi)^\ast \in \cN$. More generally there exists $\theta^\psi(\xi, \xi) \in \cNext$ determined by
 \[
 \langle \rho_{\eta, \eta}, \theta^\psi(\xi, \xi) \rangle = \left\{
 \begin{array}{ll}
 \Vert R^\psi(\xi)^\ast \eta \Vert^2_{\cH_\psi} & \eta \in \Dom(R^\psi(\xi)^\ast) \\
 \infty& \textrm{otherwise.}
 \end{array}
 \right.
 \]
 Let $\varphi$ be a normal (not necessarily semi-finite or faithful) weight on $\cN$.  We define a closed quadratic form $q_\varphi(\xi) = \langle \varphi, \theta^\psi(\xi, \xi)\rangle$ which has domain all $\xi \in D(\cH, \psi)$ such that $\theta^\psi(\xi, \xi) \in \mphi^+$. For all other $\xi$ we have $q_\varphi(\xi) = \infty$. Then the  spatial derivative $\frac{d\varphi}{ d\psi}$ is   defined as the element in $\cB(\cH)^+_{{\rm ext}}$ determined by:
\[
\langle \left( \frac{d\varphi}{ d\psi}\right)^\frac{1}{2} \xi, \left( \frac{d\varphi}{ d\psi}\right)^\frac{1}{2} \xi \rangle  =
q_\varphi(\xi).
\]
If $\varphi$ is semifinite then $\frac{d\varphi}{ d\psi}$ is an unbounded positive self-adjoint operator on $\cH$. The support of $\varphi$ equals the support of $\frac{d\varphi}{ d\psi}$. Recall also that if $\varphi$ is semi-finite and faithful -- so that $\frac{d\varphi}{d\psi}$ is invertible as an unbounded operator -- we have
\begin{equation}\label{Eqn=SpacialModular}
\begin{split}
\left( \frac{d\varphi}{d\psi} \right)^{it} x   \left( \frac{d\varphi}{d\psi} \right)^{-it} = & \sigma_t^{\varphi}( x ), \qquad x \in \cN.
\end{split}
\end{equation}
In Appendix \ref{Sect=AppendixA} we relate spatial derivatives to noncommutative $L^p$-spaces.

\section{Noncommutative $L^p$-spaces}\label{Sect=LpPrelim}

Noncommutative $L^p$-spaces associated with an arbitrary von Neumann algebra were introduced by various people. In \cite{HaaLps}, \cite{TerpI} Haagerup defined them as subspaces of weak $L^p$-spaces of the core of a von Neumann algebra (recall that the core is defined as $\cN \rtimes_{\sigma^{\varphi}} \mathbb{R}$ with   $\varphi$ a normal, semi-finite, faithful weight on $\cN$). Later Hilsum \cite{Hilsum} gave another definition based on Connes' spatial derivative \cite{Connes}. Note that Hilsum's approach in \cite{Hilsum} in fact uses Haagerup's earlier construction to show that his $L^p$-spaces form a vector space. Other approaches through the complex interpolation method can be found in the papers by Kosaki \cite{Koz}, Terp \cite{TerpII} and Izumi \cite{Izumi}.

Here we use Hilsum's definition as it is closest to the spatial theory we already studied. Of course our results translate into any of the other definitions mentioned.

\subsection{Definitions}

Let $\cN$ be a von Neumann algebra acting on a Hilbert space $\cH$. Let $\psi$ be a normal, semi-finite, faithful weight on the commutant $\cN'$. A closed densely defined operator $x$ on $\cH$ is called {\it $\gamma$-homogeneous} with $\gamma \in \mathbb{R}$ if,
\[
x a \subseteq \sigma^\psi_{i \gamma}(a) x, \qquad \textrm{ for every } a \in \cN' \textrm{ analytic for } \sigma^\psi.
\]
We now have the following theorem, see \cite{Connes}.
\begin{thm}
Let $x$ be a closed, densely defined operator on a Hilbert space $\cH$.
Write $x = u \vert x \vert$ for the polar decomposition. Then the following are equivalent:
\begin{enumerate}
\item $x$ is $(-1)$-homogeneous.
\item  $u \in \cN$ and $\vert x \vert$ is $(-1)$-homogeneous.
\item $u \in \cN$ and $\vert x \vert$ equals the spatial derivative $\frac{d\varphi}{d\psi}$ for some normal, semi-finite weight $\varphi$ on $\cN$.
\end{enumerate}
\end{thm}
In particular every positive self-adjoint $(-1)$-homogeneous operator occurs as a spatial derivative. Also note that if $x \geq 0$ is $(-1/p)$-homogeneous then $x^p$ is $(-1)$-homogeneous.   This allows us to set the following definition.
\begin{dfn}
Let $\cN$ be a von Neumann algebra. Let $\psi$ be a normal, semi-finite, faithful weight on the commutant $\cN'$. The Connes--Hilsum $L^p$-space $L^p(\cN, \psi)$ is defined as the space of all closed, densely defined $(-1/p)$-homogeneous operators $x$ such that if $x = u \vert x \vert$ is the polar decomposition then $\vert x \vert^p = \frac{d \rho}{ d\psi}$ for some $\rho \in \cN_\ast^+$.  We define,
\[
\Vert x \Vert_p = \rho(1)^{1/p}.
\]
\end{dfn}
Connes--Hilsum $L^p$-spaces are Banach spaces sharing many of the properties of classical $L^p$-spaces, such as H\"older estimates, reflexivity,  interpolation, et cetera. It is proved in \cite{TerpII} that for two choices of normal, semi-finite, faithful weights $\psi_1$ and $\psi_2$ we have that $L^p(\cN, \psi_1)$ and $L^p(\cN, \psi_2)$ are isometrically isomorphic. We will write $L^p(\cN)$ in case it is clear which weight $\psi$ is chosen.

\subsection{Interpolation} In \cite{TerpII} Terp proved that $L^p(\cN)$ is a complex interpolation space between $\cN$ and its predual $\cN_\ast$. We give a brief description here. More details on the complex interpolation method can be found in \cite{BerghLof}. We also refer   the reader to  \cite{CaspersLpf} for a slightly more elaborate discussion of what we introduce in this section.

Fix again a normal, semi-finite, faithful weight $\psi$ on $\cN'$ and denote the associated $L^p$-spaces by $L^p(\cN)$. Let $\varphi$ be a normal, semi-finite, faithful weight on $\cN$. We define the space
\[
K = \{ x \in \cN \mid \exists \varphi_x^{(-1/2)} \in \cN_\ast \textrm{ s.t. } \varphi_x^{(-1/2)}(y^\ast z) = \langle J x^\ast J \Lambda_\varphi(z), \Lambda_\varphi(y)  \rangle   \}.
\]
It is proved in \cite{TerpII} that $\mphi \subseteq K$.
By definition there is a map $j_\infty: K \hookrightarrow \cN: x \mapsto x$. Also there exists a map $j_1: K \hookrightarrow \cN_\ast: x \mapsto \varphi_x$. If we equip $K$ with the norm $\Vert x \Vert_K = \max \{ \Vert x \Vert, \Vert \varphi_x \Vert \} $ then the embeddings $j_1$ and $j_\infty$ are contractions. Dualizing them yields a commutative diagram:
\begin{equation}
 \xymatrix{
 & \cM_\ast\ar@{^{(}->}[dr]^{j_\infty^{\ast}}  & \\
   K  \ar@{^{(}->}[ur]^{j_1}\ar@{^{(}->}[dr]_{j_\infty}  & & K^\ast, \\
& \cM  \ar@{^{(}->}[ur]^{j_1^\ast} &}    \label{EqnLpIzu}
\end{equation}
Within $K^\ast$ it makes sense to speak about the intersection of $\cM$ and $\cM_\ast$ and in fact it turns out that $\cM \cap \cM_\ast = K$.
Applying the complex intpolation method at parameter $\theta \in [0, 1]$ to this diagram yields a Banach space $(\cN, \cN_\ast)_{[\theta]}$. Let $j_p: K \hookrightarrow (\cN, \cN_\ast)_{[1/p]}, p \in [1, \infty)$ be the natural inclusion.   It is proved in \cite{TerpII} that for $p \in [1, \infty)$ we have
\[
(\cN, \cN_\ast)_{[\frac{1}{p}]} \simeq L^{p}(\cN).
\]
Moreover under this isomorphism the mapping $j_1$ is given by sending $\sum_j y_j^\ast z_j$ with $y_j, z_j \in \nphi$ to
\[
\sum_j d_\varphi^{\frac{1}{2}} y_j^\ast \cdot [z_j d_\varphi^{\frac{1}{2}}], \qquad \textrm{ where } \qquad d_\varphi = \frac{d \varphi}{ d\psi}.
 \]
 We simply write $d_\varphi^{\frac{1}{2}}x   d_\varphi^{\frac{1}{2}}$ with $x = \sum_j(y_j^\ast   z_j)$ for this expression as it does not depend on the decomposition of $x$ into summands (see \cite{Lindsay} for similar considerations). For completeness we mention that in  \cite{Izumi} also embeddings of suitable subspaces of $\cN$ into $\cN_\ast$ were considered that are different from $K$. Also we note that, for $x,y \in K$ with $0 \leq x \leq y$ we have
 \begin{equation}\label{Eqn=Comparison}
 \varphi_x^{(-1/2)} \leq \varphi_y^{(-1/2)}, \qquad \textrm{ and } \qquad \varphi_x \leq \Vert x \Vert \varphi.
 \end{equation}
 For $x \in \mphi^+$ we have $\Vert \varphi_x^{(-1/2)} \Vert = \varphi(x)$.

 \begin{rmk}[Standard forms, $L^p$-spaces and the $D_\varphi$-notation] \label{Rmk=Notation}
 In this paper we shall use noncommutative $L^2$-spaces in order to explicitly write down GNS-representations of different weights on {\it the same} Hilbert space. Throughout the paper $\kappa'$ shall be a fixed normal, semi-finite, faithful weight on a von Neumann algebra $\cN'$. (In Section \ref{Sect=Ultra} onwards we may assume that $\kappa'_j$ is a fixed nsf weight on $\cM_j'$ and $\kappa'$ is a fixed nsf weight on $\cM$, with no relation between $\kappa_j'$ and $\kappa'$). We will use $L^p(\cN) = L^p(\cN, \kappa')$.
   For a normal weight $\varphi$ on $\cN$ we write
   \[
   D_\varphi = \frac{d\varphi}{d\kappa'}.
   \]
   The map $\varphi \rightarrow D_\varphi$ is order preserving.   We set,
 \[
{\rm Tr}(D_\varphi) = \varphi(1), \qquad \textrm{ for } \qquad D_\varphi \in L^1(\cN)^+,
 \]
 and extend linearly to $L^1(\cN)$. Suppose that $\varphi$ is normal, semi-finite and faithful. For $\rho := \varphi_x^{(-1/2)}$ with $x \in \mathfrak{m}_{\varphi}^+ \subseteq K$ the associated $D$-operator is given by
  \[
    D_\rho = D_{\varphi}^{\frac{1}{2}} \sqrt{x}  (D_{\varphi}^{\frac{1}{2}} \sqrt{x})^\ast,
  \]
  for which we shall write $D_{\varphi}^{\frac{1}{2}} x  D_{\varphi}^{\frac{1}{2}}$ as explained above, see \cite[Eqn. (38)]{TerpII}.
  For $z \in \nphi$ we have that the closure of $z D_{\varphi}^{\frac{1}{2}}$ is in $L^2(\cN)$. We shall again omit such  closures in the notation as elements of noncommutative $L^2$-spaces are closed by definition.  We also have that $z D_{\varphi} z^\ast = D_{z \varphi z^\ast}$, see \cite[Theorem III.14]{TerpI}. In fact we can extend our notation a bit in the following way. Suppose that $\varphi$ is a normal, semi-finite (not necessarily faithful) weight on $\cN$ and let $x \in \mathfrak{m}_{\varphi}^+$ then we write $\varphi_x^{(-1/2)}$ for the functional
  \[
  \cN \ni y \mapsto \varphi_{p_\varphi x p_\varphi}^{(-1/2)}(p_\varphi y p_\varphi)
   \]
   and note that the latter functional was defined in the corner $p_\varphi \cN p_\varphi$.
  We have that $\varphi_x^{(-1/2)} \leq \varphi_y^{(-1/2)}$ if and only if $p_\varphi x p_\varphi \leq p_\varphi y p_\varphi$.
    Finally recall that \cite[Theorem II.36]{TerpII},
 \[
 (\cN, L^2(\cN), J: x \mapsto x^\ast, L^2(\cN)^+)
 \]
  is a standard form. For the inner product we have $\langle \xi, \eta \rangle = {\rm Tr}(\eta^\ast \xi)$.
 \end{rmk}

\section{Ultraproducts of weights}\label{Sect=Ultra}

Given a series of normal weights $\varphi_j$ on von Neumann algebras $\cM_j$ we define its ultraproduct in a way that generalizes the ultraproduct of normal states. We also develop the spatial theory of ultraproducts. Consequently we are able to prove results on Tomita--Takesaki theory (such as Theorem A in the introduction).

\subsection{Groh-Raynaud ultraproducts}
Let $\omega$ be a non-principal ultrafilter on the natural numbers $\mathbb{N}$. All of the constructions in this section will work for an arbitrary set instead of $\mathbb{N}$, but we take this assumption for simplicity. Another reason for making this convention is that also   the results in \cite{AndoHaagerup} were proved under the same assumption (with again the remark that most constructions go through immediately for arbitrary sets). Indices over $\mathbb{N}$ will typically be denoted by $j$ and therefore we omit $\mathbb{N}$ in the notation from now on.

Suppose that $\mathcal{X}_j$ are Banach spaces. Then consider the space $\mathcal{A}$ of  all series $(x_j)_j$ with $x_j \in \mathcal{X}_j$ and $\sup_j \Vert x_j \Vert_{\mathcal{X}_j} < \infty$. We put a semi-norm on $\mathcal{A}$ by setting $\Vert (x_j)_j \Vert = \lim_{j, \omega} \Vert x_j \Vert_{\mathcal{X}_j}$. Let $\mathcal{A}_0$ be the space of all $x \in \mathcal{A}$ with $\Vert x \Vert = 0$. Then $\Vert \: \cdot \: \Vert$ descends to the quotient space $\mathcal{A} / \mathcal{A}_0$ and the latter space is called the {\it Banach space ultraproduct} which we denote by $\prod_{j, \omega} \mathcal{X}_j$. The equivalence class of a sequence $(x_j)_j$ in $\prod_{j, \omega} \mathcal{X}_j$ will from this point be denoted by $(x_j)_\omega$. This notation is taken from \cite{AndoHaagerup}. Note however that \cite{Raynaud} uses the notation $(x_j)^\bullet$ for $(x_j)_\omega$.

If each $\mathcal{X}_j$ is a Hilbert space then so is $\prod_{j, \omega} \mathcal{X}_j$ in a natural way. If each $\mathcal{X}_j$ is a C$^\ast$-algebra then so is $\prod_{j, \omega} \mathcal{X}_j$. It is not true that if each $\mathcal{X}_j$ is a von Neumann algebra then the Banach space ultraproduct $\prod_{j, \omega} \mathcal{X}_j$ is a von Neumann algebra. Therefore we define ultraproducts of von Neumann algebras differently.

We  recall the {\it Groh-Raynaud} ultraproduct of von Neumann algebras. First note that if $\bA_j$ is a C$^\ast$-algebra acting on a Hilbert space $\cH_j$ then $(x_j)_\omega \in \prod_{j, \omega} \bA_j$ acts naturally on $\prod_{j, \omega} \cH_j$ by $(x_j)_\omega (\xi_j)_\omega = (x_j \xi_j)_\omega$.
\begin{enumerate}
\item For a family of von Neumann algebras $\cM_j$ that are represented on a Hilbert space $\cH_j$ we define the {\it abstract ultraproduct} $\prod_{j, \omega} (\cM_j, \cH_j)$ as the closure in the strong operator topology of the $\ast$-algebra of operators $(x_j)_\omega$ acting on $\prod_{j, \omega} \cH_j$.
\item For a family of von Neumann algebras $\cM_j$ we define the {\it Groh-Raynaud ultraproduct} $\prod_{j, \omega} \cM_j$ as $\prod_{j, \omega} (\cM_j, \cH_j)$ where $\cH_j$ is the standard form Hilbert space.
\end{enumerate}

As shown in \cite[Remark 3.6]{AndoHaagerup} the strong closure in the above definitions is necessary. In this paper if we write $x = (x_j)_\omega \in \cM$ we mean that $x \in \cM$ is an element that can be written as a sequence $(x_j)_\omega$. That is, $x$ is contained in the Banach space ultraproduct of $\cM_j$.

Recall that besides the Groh-Raynaud ultraproduct there are other definitions of ultraproducts (such as the Ocneanu ultraproduct) which will not be used in this paper. A complete and detailed account of different ultraproducts and their relations was given by Ando--Haagerup in \cite{AndoHaagerup}.

It is proved in \cite{Raynaud} that there is a natural pairing between $\prod_{j, \omega} \cM_j$ (ultraproduct of von Neumann algebras) and $\prod_{j, \omega} (\cM_j)_\ast$ (ultraproduct of Banach spaces) determined by  $\langle (\rho_j)_{\omega}, (x_j)_{\omega}\rangle = \lim_{j, \omega} \rho_j(x_j)$ and moreover through this pairing $\prod_{j, \omega} \cM_j = \left( \prod_{j, \omega} (\cM_j)_\ast \right)^\ast$. Recall also \cite{AndoHaagerup} that for the commutants we have $\left(\prod_{j, \omega} \cM_j \right)' = \prod_{j, \omega} \cM_j'$. We use these facts without further reference.

\subsection{Ultraproducts of unbounded operators}\label{Sect=Cone}
In this section we define the ultraproduct of elements in the extended positive cone $(\cM_j)_{{\rm ext}}^+$. We do this using a bounded transform. We define the continuous increasing and operator monotone function
\[
F(s) = \frac{s}{1 + s}, \qquad s \in [0, \infty).
\]
For $x \in \cMext$ with decomposition \eqref{Eqn=Integral} there is a unique $F(x) \in \cM^+$ given by
\[
\langle F(x), \rho \rangle_{\cM, \cM_\ast} = \int_{0}^\infty  F(\lambda) \:d\rho(e_\lambda) + \rho(p).
\]
Conversely for $x \in \cM^+$ with $\Vert x \Vert \leq 1$ and spectral decomposition $x = \int_0^\infty \lambda \:\: de_\lambda$ we set $F^{-1}(x) \in \cMext$ by
\[
\langle F^{-1}(x), \rho \rangle_{\cM, \cM_\ast} = \int_{0}^\infty  F^{-1}(\lambda) \:d\rho(e_\lambda) +  \rho(p) \cdot \infty,
\]
where $p = \int_0^\infty \delta_{1}(\lambda) de_\lambda$, i.e. the projection on the eigenspace of the eigenvalue 1.  As $F$ is operator monotone and strongly continuous it preserves suprema of bounded increasing nets. Now if $\cM_j$ are von Neumann algebras and $x_j \in \cM_{j, {\rm ext}}^+$ then we set its ultrapower by
\begin{equation}\label{Eqn=UltraProductExtPos}
(x_j)_\omega = F^{-1}((F(x_j))_\omega).
\end{equation}
Note that if   $x_j \in \cM_{j, {\rm ext}}^+$ is contained in $\cM^+$ and its uniform norm converges to 0 then \eqref{Eqn=UltraProductExtPos} equals zero so that our notation $(x_j)_\omega$ is unambiguous.

We need some results on spectral calculus on ultraproducts. Firstly for a continuous function $f$ and $x = (x_j)_\omega \in \cM$ we have
\[
f((x_j)_\omega) = (f(x_j))_\omega, \qquad (x_j)_\omega \in \cM.
\]
Indeed this can be seen by approximating $f$ with polynomials. In the unbounded situation we need two lemmas.

\begin{lem}\label{Lem=FunctionaCalculus}
Let $f: [0, \infty) \rightarrow [0, \infty)$ be a continuous function for which $f(t)$ has a limit in $[0, \infty]$ as $t \rightarrow \infty$. Let $x_j \in (\cM_j)_{{\rm ext}}^+$. Then $f( (x_j)_\omega ) = (f(x_j))_\omega$.
\end{lem}
\begin{proof}
As the limit $f(t), t \rightarrow \infty$ exists we may view $F \circ f \circ F^{-1}$ as a continuous function $[0,1] \rightarrow [0,1]$. Then we have,
\[
\begin{split}
& f( (x_j)_\omega )
=   f( F^{-1} (  ( F(x_j)  )_\omega )  )
= F^{-1} \circ (F \circ f \circ F^{-1})  (  F(x_j)_\omega ) \\
= & F^{-1}    ( ( (F \circ f \circ F^{-1}) \circ F(x_j))_\omega )  =  F^{-1}    ( (F(f(x_j)))_\omega ) = (f(x_j))_\omega.
\end{split}
\]
\end{proof}

The statement of the previous lemma does not hold for the function $f(t) = e^{it}$, see \cite[Example 4.7]{AndoHaagerup}. However if we restrict to a suitable spectral subspace we get the following lemma. We use $\chi_A$ for the indicator function on a set $A$.

\begin{lem}[Lemma 4.4 of \cite{AndoHaagerup}]\label{Thm=SpectralStuff}
Let $x_j$ be positive self-adjoint operators affiliated with a von Neumann algebra $\cM_j$.  Let $x = (x_j)_\omega \in \cMext$ be its ultrapower. Let $\mathcal{K} \subseteq \mathcal{H}$ be the space given by $\cup_{\lambda> 0}\chi_{(\frac{1}{\lambda}, \lambda)}(x) \cH$. Then for every $t \in \mathbb{R}$,
\[
(x\vert_{\mathcal{K}})^{it} = (x_j^{it})_\omega \vert_{\mathcal{K}}.
\]
\end{lem}

For an operator $A \in \cB(\cH)_{{\rm ext}}^+$ we set $P_{\infty}(A)$ for the projection   $p$ defined in the spectral resolution \eqref{Eqn=Integral}. We also set $P_{<\infty}(A) = 1-p$. The next lemma shows that we may compute values of ultraproducts in a natural way.

\begin{lem}\label{Lem=UltraOfElements}
Let $A_j$ be (unbounded) positive self-adjoint operators acting on a Hilbert space $\cH_j$ with domain $\Dom(A_j)$. Let $A:=(A_j)_\omega$ be its ultraproduct which is contained in $\cB(\cH)_{{\rm ext}}^+$ with $\cH = \prod_{j, \omega} \cH_j$. Let $\xi_j \in \Dom(A_j)$ and suppose that $\Vert \xi_j \Vert_2$ is bounded so that we may put $\xi = (\xi_j)_\omega \in \cH$. We have,
\begin{equation}\label{Eqn=SubtleChange}
\left< A, \omega_{\xi,\xi} \right> \leq \lim_{j, \omega} \Vert A_j^{\frac{1}{2}} \xi_j \Vert_2^2.
\end{equation}
Moreover, let $\xi \in P_{<\infty}(A) \cH$ be such that $\sup_j \Vert   A_j^{\frac{1}{2}} \xi_j \Vert_2^2$ is finite. Then  we get an equality,
\begin{equation}\label{Eqn=PIsNeeded}
  A^{\frac{1}{2}} (\xi_j)_\omega =  P_{< \infty}(A) (A_j^{\frac{1}{2}} \xi_j)_\omega.
\end{equation}
\end{lem}
\begin{proof}
Let $G_\lambda$ be a continuous function supported on $[0, 2 \lambda]$ such that $0 \leq G_\lambda \leq 1$ and $G_\lambda(s) = 1$ for all $s \in [0, \lambda]$. The projection $P_{<\infty}(A)$ is given by the supremum of the projections $G_\lambda(A), \lambda \geq 0$.  Write $\xi_1 = P_{< \infty}(A) \xi, \xi_2 = (1- P_{< \infty}(A)) \xi$. Then, writing $\xi_1 = (\xi_{1,j})_\omega$,
\begin{equation}\label{Eqn=SubtleComp}
\begin{split}
& \langle A, \rho_{\xi_1,\xi_1} \rangle =
\sup_{\lambda > 0} \: \langle A, \rho_{G_\lambda(A) \xi_1, G_\lambda(A)\xi_1} \rangle
=   \sup_{\lambda>0} \:  \langle (A_j  G_\lambda(A_j))_\omega, \rho_{\xi_1, \xi_1} \rangle \\
= &
\sup_{\lambda>0} \:  \lim_{j, \omega} \Vert  A_j^{\frac{1}{2}}   G_\lambda(A_j)  \xi_{1,j} \Vert_2^2
\leq   \lim_{j, \omega} \Vert  A_j^{\frac{1}{2}}    \xi_{1,j} \Vert_2^2.
\end{split}
\end{equation}
This shows \eqref{Eqn=SubtleChange} in case $\xi_2 = 0$.
If $\xi_2 \not = 0$ then write $\xi_2 = (\xi_{2, j})_\omega$. There exists $0 < \epsilon < \Vert \xi_2 \Vert$ such that for every $\lambda > 0$ there exists $U \in \omega$ such that for all $j \in U$ we have $\Vert \chi_{(\lambda, \infty)}(A_j) \xi_{2, j} \Vert > \Vert \xi_2 \Vert - \epsilon$.  But then clearly  $\lim_{j, \omega} \Vert A_j^{\frac{1}{2}} \xi_j \Vert_2^2 = \infty$. The inequality  \eqref{Eqn=SubtleChange} then follows.

It remains to prove the final statement.  Take $\xi = (\xi_j)_\omega \in P_{<\infty}(A) \cH$, suppose that $\sup_j \Vert   A_j^{\frac{1}{2}} \xi_j \Vert_2^2$ is finite. Firstly note that $A$ acts on $P_{< \infty}(A) \cH$ as an unbounded operator and that the assumptions ensure that $\xi$ is in the domain of $A^{\frac{1}{2}}$ by \eqref{Eqn=SubtleChange}. Note that (as in \eqref{Eqn=SubtleComp}) we get again,
\[
\begin{split}
&  P_{< \infty}(A) (A_j^{\frac{1}{2}} \xi_j)_\omega = \lim_{\lambda \rightarrow \infty} G_\lambda(A) (A_j^{\frac{1}{2}} \xi_j)_\omega \\
= & \lim_{\lambda \rightarrow \infty} ( G_\lambda(A_j) A_j^{\frac{1}{2}} \xi_j)_\omega = \lim_{\lambda \rightarrow \infty} G_\lambda(A)A^{\frac{1}{2}} (\xi_j)_\omega = A^{\frac{1}{2}} (\xi_j)_\omega.
\end{split}
\]
 This concludes the proof.
\end{proof}

Note that in \eqref{Eqn=PIsNeeded} the projection $P_{<\infty}(A)$ is necessary. For example, let $\cH = L^2(\mathbb{R})$. Take $\xi_j = j^{-1/2} \chi_{[j, j+1]}$ and let $A_j$ be the multiplication operator with the identify function $s \mapsto s$. Then $(\xi_j)_\omega = 0$ whereas $(A_j \xi_j)_\omega \not = 0$.

\subsection{Ultraproducts of weights}

For two normal weights $\rho$ and $\varphi$ on a von Neumann algebra $\cM$ we say that $\rho \leq \varphi$ if $\rho(x) \leq \varphi(x)$ for every $x \in \cM^+$. For an increasing net $(\rho_k)_k$ of normal weights we say that $\varphi = \sup_k \rho_k$ if $\varphi$ is the smallest weight majorizing each $\rho_k$.  In fact for an increasing net $(\rho_k)_k$ with $\rho_k \in \cM_\ast$  we may  define a normal weight $\varphi$ as,
\[
\varphi(x) = \sup_k \langle \rho_k, x \rangle_{\cM_\ast, \cM}, \qquad x \in \cM^+.
\]
Note that $\varphi$ is  not necessarily semi-finite. If the net $(\rho_k)_k$ were to be bounded then $\varphi \in \cM_\ast$ and the convergence $\rho_k \rightarrow \varphi$ is in  norm (as $\Vert \rho_k - \varphi \Vert = \varphi(1) - \rho_k(1) \rightarrow 0$).

Now we recall the following theorem, which is a strengthening of \cite[Theorem VII.1.11 (iv)]{TakII}. We give a proof here.

\begin{thm}\label{Thm=NetStructure}
Let $\varphi$ be a normal, semi-finite, faithful weight on a von Neumann algebra $\cN$. Consider the set
\[
\mathcal{F}_\varphi = \left\{ \rho \in \cN_\ast \mid \rho \leq \varphi \right\}.
\]
Then $\varphi = \sup_{ \rho \in \mathcal{F}_\varphi}  \rho$.
\end{thm}
\begin{proof}
We need to prove that $\mathcal{F}_\varphi$ is a net in $\cN_\ast$ (note that this is stronger than \cite[Theorem VII.1.11]{TakII}). Take $\rho \in \cN_\ast$ with $\rho \leq \varphi$. Let $D_\rho$ and $D_\varphi$ be the associated densities. As $D_\rho \leq D_\varphi$ we have $D_\varphi^{-\frac{1}{2}} D_\rho D_\varphi^{-\frac{1}{2}} \leq 1$. In fact $x_{\rho} := [D_\varphi^{-\frac{1}{2}} D_\rho D_\varphi^{-\frac{1}{2}}] \in \cN$ as $x_\rho$ is $0$-homogeneous. We claim moreover that $x \in K$ (see Section \ref{Sect=LpPrelim}). Indeed, take $y,z \in \nphi$. As,
\[
\begin{split}
[z D_\varphi^{\frac{1}{2}}] \cdot [D_\varphi^{-\frac{1}{2}}   D_\rho  D_\varphi^{-\frac{1}{2}} ] \cdot  D_\varphi^{\frac{1}{2}} y^\ast
\supseteq  z D_\varphi^{\frac{1}{2}} D_\varphi^{-\frac{1}{2}}   D_\rho  D_\varphi^{-\frac{1}{2}}    D_\varphi^{\frac{1}{2}} y^\ast
\supseteq z D_\rho y^\ast,
\end{split}
\]
we have that $[z D_\varphi^{\frac{1}{2}}] [D_\varphi^{-\frac{1}{2}}   D_\rho  D_\varphi^{-\frac{1}{2}} ]   D_\varphi^{\frac{1}{2}} y^\ast \supseteq z  \cdot D_\rho  \cdot y^\ast$ and as $z   \cdot D_\rho  \cdot y^\ast \in L^1(\cN)$ the inclusion is actually an equality by \cite[Theorem 4]{Hilsum}. Then,
\begin{equation}\label{Eqn=ReverseThing}
\begin{split}
 & y^\ast z \mapsto \langle J x_\rho^\ast J \Lambda(z), \Lambda(y) \rangle
=   \langle x_\rho J \Lambda(y), J \Lambda(z) \rangle
=  {\rm Tr}( [z D_\varphi^{\frac{1}{2}}]\cdot x_\rho \cdot D_\varphi^{\frac{1}{2}} y^\ast )\\
= & {\rm Tr}( [z D_\varphi^{\frac{1}{2}}] \cdot [D_\varphi^{-\frac{1}{2}}   D_\rho  D_\varphi^{-\frac{1}{2}} ]    \cdot D_\varphi^{\frac{1}{2}} y^\ast )
= {\rm Tr}( z \cdot D_\rho  \cdot y^\ast ) = \rho(y^\ast z),
\end{split}
\end{equation}
extends boundedly to the normal functional $\varphi_{x_\rho}^{(-1/2)} = \rho$, so that $x_\rho \in K$.
 Now let $\rho_1, \rho_2 \in \mathcal{F}_\varphi$. Let $x_{\rho_1}$ and $x_{\rho_2}$ be the associated operators as before. The function $F(s) = s (1+s)^{-1}$ is operator monotone and operator concave. We let $x \in \cN^+$ be such that  $F^{-1}(x) = F^{-1}(x_{\rho_1}) +  F^{-1}(x_{\rho_2})$  (note that this equality takes values in $\cN^+_{{\rm ext}}$ on which we have well-defined spectral calculus).  We have $x_{\rho_1} \leq x, x_{\rho_2} \leq x$ by operator monotonicity of $F$ and further $x \leq 1$. Furthermore as $F$ is operator concave and $F(2F^{-1}(s)) = 2s(1+s)^{-1} \leq 2, s \in \mathbb{R}$, we have,
  \[
  \begin{split}
  x = & F(  F^{-1}(x_{\rho_1}) +  F^{-1}(x_{\rho_2})  )
    = F\left( \frac{1}{2}(  2F^{-1}(x_{\rho_1}) +  2F^{-1}(x_{\rho_2})  )  \right)\\
    \leq  &  F(   2F^{-1}(x_{\rho_1}) ) +  F( 2 F^{-1}(x_{\rho_2})  )
    \leq 2 (x_{\rho_1} + x_{\rho_2}).
\end{split}
\]
So $x \in K$.
 Therefore for $i= 1,2$ we get $\varphi^{(-1/2)}_{ x_{\rho_i}} \leq  \varphi^{(-1/2)}_{x} \leq \varphi$. This  shows that $\mathcal{F}_\varphi$ is a net. Taking an  net $e_j \in \mathfrak{m}_\varphi^+, e_j \leq 1$ converging to $1$ strongly and such that $\Vert e_j \Vert \leq 1$ gives that $\varphi_{e_j}^{(-1/2)}  \nearrow\varphi$. So $\varphi = \sup_{ \rho \in \mathcal{F}_\varphi}  \rho$.
\end{proof}

\begin{dfn} \label{Dfn=Ultraproduct}
As before let $\cM_j$ be von Neumann algebras and let $\cM = \prod_{j, \omega} \cM_j$ be its ultraproduct.
Let $\varphi_j$ be  a normal weight on the von Neumann algebra $\cM_j$. Consider the set $\mathcal{F} := \mathcal{F}_{\varphi_j, j \in \mathbb{N}}$ of all ultraproducts $(\rho_j)_j \in \cM_\ast$ with $\rho_j \in \mathcal{F}_{\varphi_j}$ and $\Vert \rho_j \Vert$ bounded in $j$. Then set $\varphi = \sup_{\rho \in \mathcal{F}} \rho$.  We call $\varphi$ the {\it ultraproduct weight} and we denote it by $(\varphi_j)_\omega$. Clearly $\varphi$ is normal.
\end{dfn}

\begin{rmk}\label{Rmk=AndoHaagerupWeight}
In \cite{AndoHaagerup}   another ultraproduct of  weights was considered. Suppose that $\cN$ is a von Neumann algebra equipped with a fixed normal weight $\varphi_\cN$. Let $\cM_j = \cN$ and let $\cM = \prod_{j, \omega} \cN$ be its ultraproduct. There exists a conditional expectation $\mathcal{E}: \cM \rightarrow \cN: (x_j)_\omega \mapsto \lim_{j, \omega} x_j$, where the limit is understood in the $\sigma$-weak topology. The normal weight $\varphi_\cM = \varphi_\cN \circ \mathcal{E}$ is in \cite{AndoHaagerup} defined as the ultraproduct of the fixed series of weights $\varphi_\cN$. This definition does not agree with ours. Take for example $\cN = L^\infty(\mathbb{R})$ with weight $\varphi_\cN$ the integral against the standard Lebesgue measure. Take $f_j$ the indicator function on $[j, j+1]$. Then $f_j \rightarrow 0$ in the $\sigma$-weak topology. So $\varphi_\cN \circ \mathcal{E}((f_j)_\omega) = 0$. Whereas $\langle (\varphi)_\omega, (f_j)_\omega \rangle = 1$. It is easy to check that in general $\varphi_\cM \leq (\varphi)_\omega$.
\end{rmk}

\subsection{Spatial theory of ultraproducts}
Let again  $\cM_j$ be a sequence of  von Neumann algebras and let $(\cM_j, \cH_j, J_j, P_j)$
be its standard form which we may and will identify (uniquely, see discussion below Definition \ref{Dfn=Standard}) with $(\cM_j, L^2(\cM_j), J_j, L^2(\cM_j)^+)$.
Let $\cM = \prod_{j, \omega} \cM_j$ be its Groh-Raynaud ultrapower which acts on the ultraproduct Hilbert space $\cH = \prod_{j, \omega} \cH_j$. By \cite{AndoHaagerup} $(\cM, \cH, J_\omega, \mathcal{P})$ is a standard form where $\mathcal{P}$ is the positive cone given by  the vectors $(\xi_j)_\omega$ with $\xi_j \in \mathcal{P}_j$. Therefore we may and will identify  $(\cM, \cH, J_\omega, \mathcal{P})$ with $(\cM, L^2(\cM), J: x \mapsto x^\ast, L^2(\cM)^+)$ and again there is a unique such identification.  Recall also that for a normal, semi-finite, faithful weight $\varphi_j$ on $\cM_j$ the map $z \mapsto z D_{\varphi_j}^{\frac{1}{2}}$ is a GNS-construction and the analogous result holds for weights on $\cM$.
  If $\rho = (\rho_j)_\omega$ is a vector in $\cM_\ast$ we get vectors $D_\rho^{\frac{1}{2}}$ and $D_{\rho_j}^{\frac{1}{2}}$ in the respective positive cones $\mathcal{P}$ and $\mathcal{P}_j$. Then $D_{\rho}^{\frac{1}{2}} = (D_{\rho_j}^{\frac{1}{2}})_\omega$. In the weight case we get Lemma \ref{Lem=OpenLemma} below.

For the next lemma we use the notation $\Vert z \Vert_{2, \varphi} = \varphi(z^\ast z )$.
 Also recall again the conventions on closures of operators in $L^p$-spaces from Remark \ref{Rmk=Notation}.

\begin{lem}
Let $\cN$ be a von Neumann algebra. Suppose that $(\rho_k)_{k \in K}$ is an increasing net in $\cN_\ast$ with supremum a normal, semi-finite weight $\varphi$. Let $z \in \nphi$. Then, $\Vert z D_{\rho_k}^{\frac{1}{2}} - z D_{\varphi}^{\frac{1}{2}}\Vert_2 \rightarrow 0$.
\end{lem}
\begin{proof}
Without loss of generality we may assume that $\varphi$ is faithful, because if it is not then we may consider $\varphi + \varphi'$ instead where $\varphi$ is a normal, semi-finite weight with support equal to $1- p_\varphi$ and for which $z \in \mathfrak{n}_{\varphi'}$. Then $\varphi'  = \sup_{\rho' \in \mathcal{F}_{\varphi'}} \rho'$ so that the $\rho_k + \rho'$ with $k \in K$ and $\rho' \in \mathcal{F}_{\varphi'}$ form an increasing net converging to $\varphi + \varphi'$. Then, assuming that we can prove the lemma for $\varphi + \varphi'$, we have $\Vert z D_{\rho_k + \rho'}^{\frac{1}{2}} - z D_{\varphi + \varphi'}^{\frac{1}{2}} \Vert_2 \rightarrow 0$. By orthogonality,
\[
\Vert  z D_{\rho_k + \rho'}^{\frac{1}{2}} - z D_{\varphi + \varphi'}^{\frac{1}{2}} \Vert_2^2 =
\Vert z D_{\rho_k}^{\frac{1}{2}} - z D_{\varphi}^{\frac{1}{2}} \Vert_2^2 + \Vert z D_{\rho'}^{\frac{1}{2}} - z D_{\varphi}^{\frac{1}{2}} \Vert_2^2
\]
so that  $\Vert z D_{\rho_k }^{\frac{1}{2}} - z D_{\varphi}^{\frac{1}{2}} \Vert_2 \rightarrow 0$.

Because $\varphi$ is faithful we may consider its Tomita algebra $\mathcal{T}_\varphi$. We first prove the lemma under the assumption that $z \in \mathcal{T}_{\varphi}$.
As $\rho_k \leq \varphi$ we have that $D_{\rho_k}^{\frac{1}{2}} \leq D_\varphi^{\frac{1}{2}}$. So that $x_{\rho_k} =  D_\varphi^{-\frac{1}{4}} D_{\rho_k}^{\frac{1}{2}} D_\varphi^{-\frac{1}{4}}$ is an element of $\cN$ with norm $\leq 1$. Moreover $(x_{\rho_k})_{k \in K}$ is an increasing net with supremum equal to $1$ (see the proof of Theorem \ref{Thm=NetStructure}). In particular $x_{\rho_k} \rightarrow 1$ in the strong topology. Now,
\begin{equation}\label{Eqn=StartConvergence}
\begin{split}
& \Vert z D_{\rho_k}^{\frac{1}{2}} - z D_\varphi^{\frac{1}{2}} \Vert_2^2 \\
= &
\langle   z D_{\rho_k}^{\frac{1}{2}},  z D_{\rho_k}^{\frac{1}{2}}  \rangle +
\langle  z D_\varphi^{\frac{1}{2}}, z D_\varphi^{\frac{1}{2}}   \rangle -
\langle  z D_\varphi^{\frac{1}{2}},  z D_{\rho_k}^{\frac{1}{2}}  \rangle -
\langle  z D_{\rho_k}^{\frac{1}{2}}  ,  z D_\varphi^{\frac{1}{2}}  \rangle \\
= & \langle   z  D_{\rho_k}^{\frac{1}{2}},  z D_{\rho_k}^{\frac{1}{2}}  \rangle +
\langle  z D_\varphi^{\frac{1}{2}}, z D_\varphi^{\frac{1}{2}}   \rangle -
\langle  z D_\varphi^{\frac{1}{2}},  z D_{\varphi}^{\frac{1}{4}}  x_{\rho_k} D_{\varphi}^{\frac{1}{4}} \rangle -
\langle  z D_{\varphi}^{\frac{1}{4}}  x_{\rho_k} D_{\varphi}^{\frac{1}{4}} ,  z D_\varphi^{\frac{1}{2}}  \rangle. \\
\end{split}
\end{equation}
We have, c.f. Remark \ref{Rmk=Notation},
\[
\begin{split}
& \langle   z  D_{\rho_k}^{\frac{1}{2}},  z D_{\rho_k}^{\frac{1}{2}}  \rangle  =
 {\rm Tr}\left(     z      D_{\rho_k}   z^\ast  \right) =  {\rm Tr}\left(       D_{z \rho_k z^\ast}      \right) \rightarrow  {\rm Tr}\left(         D_{z \varphi z^\ast}    \right) =   {\rm Tr}\left(     z    D_{\varphi} z^\ast    \right)  = \langle   z  D_{\varphi}^{\frac{1}{2}},  z D_{\varphi}^{\frac{1}{2}}   \rangle.
\end{split}
\]
Then we get as $x_{\rho_k} \rightarrow 1$ strongly (and as $z \in \mathcal{T}_\varphi$ each of the following expressions is well defined),
\[
\begin{split}
& \langle  z D_\varphi^{\frac{1}{2}},  z D_{\varphi}^{\frac{1}{4}}  x_{\rho_k} D_{\varphi}^{\frac{1}{4}} \rangle =
  {\rm Tr}\left(  D_{\varphi}^{\frac{1}{4}}  x_{\rho_k} D_{\varphi}^{\frac{1}{4}}   z^\ast z D_\varphi^{\frac{1}{2}} \right) \\
  = &
    {\rm Tr}\left(    x_{\rho_k} D_{\varphi}^{\frac{1}{4}}   z^\ast z D_\varphi^{\frac{3}{4}} \right)
    \rightarrow
     {\rm Tr}\left(   D_{\varphi}^{\frac{1}{4}}   z^\ast z D_\varphi^{\frac{3}{4}} \right)
     =
\langle  z D_\varphi^{\frac{1}{2}},  z D_{\varphi}^{\frac{1}{2}}   \rangle.
\end{split}
\]
And similarly,
\[
\langle  z D_{\varphi}^{\frac{1}{4}}  x_{\rho_k} D_{\varphi}^{\frac{1}{4}} ,  z D_\varphi^{\frac{1}{2}}  \rangle
\rightarrow \langle  z D_{\varphi}^{\frac{1}{2}}   ,  z D_\varphi^{\frac{1}{2}}  \rangle.
\]
 This shows that \eqref{Eqn=StartConvergence} converges to 0.

Next we treat a general $z \in \nphi$. Let $z_l \in \mathcal{T}_\varphi$ be such that $\Vert z - z_l \Vert_{2, \varphi} \rightarrow 0$. Recall again that $\Vert z - z_l \Vert_{2, \rho_k} \leq \Vert z - z_l \Vert_{2, \varphi}$. So let $\epsilon > 0$ and take $l$  such that  $\Vert z - z_l \Vert_{2, \varphi} < \epsilon $. Then let $k_0$ be such that for $k \geq k_0$ we have $\Vert z_l D_{\rho_k}^{\frac{1}{2}} - z_l D_{\varphi}^{\frac{1}{2}} \Vert  < \epsilon$. We get that,
\[
\Vert z D_{\rho_k}^{\frac{1}{2}} - z D_{\varphi}^{\frac{1}{2}} \Vert_2 \leq
\Vert z D_{\rho_k}^{\frac{1}{2}} - z_l D_{\rho_k}^{\frac{1}{2}} \Vert_2
+ \Vert z_l D_{\rho_k}^{\frac{1}{2}} - z_l D_{\varphi}^{\frac{1}{2}} \Vert_2
+ \Vert z_l D_{\varphi}^{\frac{1}{2}} - z D_{\varphi}^{\frac{1}{2}} \Vert_2 < 3 \epsilon.
\]
This shows that $\Vert z D_{\rho_k}^{\frac{1}{2}} - z D_{\varphi}^{\frac{1}{2}}\Vert_2 \rightarrow 0$.
\end{proof}

\begin{lem}\label{Lem=Atom}
For every $j$ let $p_j$ be an atomic projection in $\cM_j$. Then $(p_j)_\omega$ is atomic in $\cM$.
\end{lem}
\begin{proof}
The set $(p_j)_\omega \cM (p_j)_\omega$ is the closure of all series $(p_j)_\omega (x_j)_\omega  (p_j)_\omega = (p_j x_j p_j)_\omega$ where $(x_j)_\omega \in \cM$. So it equals the Banach space ultraproduct $\prod_{j, \omega} p_j \cM_j p_j \simeq \prod_{j, \omega} \mathbb{C} \simeq \mathbb{C}$.
\end{proof}

\begin{lem}\label{Lem=TakingTheSup}
Let $\varphi = (\varphi_j)_\omega$ be an ultrapower of normal, semi-finite, faithful weights $\varphi_j$. Assume that $\varphi$ is semi-finite. For $\rho_j \in \mathcal{F}_{\varphi_j} \subseteq (\cM_j)_\ast$ set $x_{\rho_j} = D_{\varphi_j}^{-\frac{1}{2}} D_{\rho_j} D_{\varphi_j}^{-\frac{1}{2}}$. We have that the supremum over $\rho = (\rho_j)_\omega \in \mathcal{F} := \mathcal{F}_{\varphi_j, j \in \mathbb{N}}$ of $(x_{\rho_j})_\omega$ is 1.
\end{lem}
\begin{proof}
Firstly note that the elements $(x_j)_\omega$ with $x_j \in \mathfrak{m}_{\varphi_j}^+$ and $x_j \leq 1$ form an increasing net that converges to 1 (see the proof of Theorem \ref{Thm=NetStructure}). Recall that for $x_j \in \mathfrak{m}_{\varphi_j}^+$ we have $x_j = x_{\rho_{x_j}}$ where $\rho_{x_j} := \varphi_{j,x_j}^{(-1/2)}$ (see \eqref{Eqn=ReverseThing}), but there is not necessarily a uniform bound on the $\rho_{x_j}$.
Now fix some $(x_j)_\omega$ with $x_j \in \mathfrak{m}_{\varphi_j}^+$ and $x_j \leq 1$.  Take a spectral decomposition $x_j = \int \lambda \: dE_{x_j}(\lambda)$. Note that the support of the spectral measure is contained in $[0,1]$. Let $S_j \subseteq (0,1]$ be the set of (non-zero) atoms, i.e. the set of $s \in (0,1]$ such that $\int \delta_s dE_{x_j}(\lambda)$ is non-zero, say $x_{j, s}$. For later use let $\widetilde{x}_j = x_j - \sum_{s \in S_j} x_{j, s}$ be the non-atomic part of $x_j$.
We decompose $x_{j, s}$ further into $x_{j, s} = \sum_{k \in K_{j,s}} x_{j, s, k}$ in such a way that either $\Vert \rho_{x_{j, s, k}}  \Vert \leq 1$ or  $\frac{1}{\Vert x_{j, s, k} \Vert} x_{j, s, k}$ is an atomic projection in $\cM_j$. For the non-atomic part for every $j$ we may take closed intervals $I_{j,t} \subseteq (0, 1]$ indexed by $t$ such that $y_{j,t} = \int_{I_{j,t}} \lambda \: dE_{\widetilde{x}_j}(\lambda)$ has the property that $\Vert \rho_{y_{j,t}}\Vert \leq 1$.

Now let $\mathcal{A}$ be the set of all ultraproducts $(z_j)_\omega$ where $z_j$ is one of the $x_{j,s,k}$'s or $y_{j,t}$'s defined above. We claim that $\lim_{j, \omega} \varphi_j(z_j) < \infty$. Indeed if this were not to be the case then for every $C> 1$ there exists $U \in \omega$ such that for every $j \in U$ we have $\varphi_j(z_j) > C$. But this means that for every $j \in U$, we had chosen $z_j$ to be a multiple of an atomic projection, i.e. $p_j := \frac{1}{\Vert z_j \Vert} z_j$. Then $(p_j)_\omega$ is an atomic projection on $\cM$ (see Lemma \ref{Lem=Atom}) and $(p_j)_\omega \cM (p_j)_\omega \simeq \mathbb{C}$. But then $\langle (\varphi_j)_\omega, (p_j)_\omega \rangle = \lim_{j, \omega} \varphi_j(p_j) >  C$. We may therefore take decreasing sets $U_1 \supseteq U_2 \supseteq \ldots$ in $\omega$ such that we may pick for every $j \in U_k$ an atomic projection $p_j$ such that $\varphi_j(p_j) > k$. Then $\langle (\varphi_j)_\omega, (p_j)_\omega \rangle = \lim_{j, \omega} \varphi_j(p_j) = \infty$  which contradicts that $\varphi$ is semi-finite (as $(p_j)_\omega$ is atomic).

 The previous paragraph shows that for $(z_j)_\omega \in \mathcal{A}$ we have that $\rho_{z_j}$ is bounded in $j$ (as $\Vert \rho_{z_j} \Vert = \varphi_j(z_j)$) and so $(\rho_{z_j})_\omega \in \mathcal{F}$. By construction $(x_j)_\omega$ is the smallest element that is bigger then all $(z_j)_\omega \in \mathcal{A}$. Then as all elements  $(x_j)_\omega$ with $x_j \in \mathfrak{m}_{\varphi_j}^+$ have supremum 1 we see that also the    supremum of $(x_{\rho_j})$ with $(\rho_j)_\omega \in \mathcal{F}$ equals 1.
\end{proof}

 \begin{lem}\label{Lem=OpenLemma}
 Let $\varphi = (\varphi_j)_\omega$ be an ultrapower of normal, semi-finite, faithful weights $\varphi_j$. Assume that $\varphi$ is semifinite.  Let $z_j \in \mathfrak{n}_{\varphi_j}$ be bounded in $j$ such that also $\Vert z_j \Vert_{2, \varphi_j}$ is bounded. Then $(z_j)_\omega \in \mathfrak{n}_{\varphi}$ and,
 \[
 (z_j D_{\varphi_j}^{\frac{1}{2}})_\omega = (z_j)_\omega D_\varphi^{\frac{1}{2}}.
 \]
 \end{lem}
 \begin{proof}
 Firstly, with $\mathcal{F}$ as in Definition \ref{Dfn=Ultraproduct}, as,
 \[
 \langle (\varphi_j)_\omega, (z_j^\ast)_\omega (z_j)_\omega \rangle = \sup_{\rho \in \mathcal{F}} \lim_{j, \omega} \langle \rho_j, z_j^\ast z_j \rangle \leq \lim_j \langle \varphi_j, z_j^\ast z_j \rangle < \infty,
 \]
 it follows that $(z_j)_\omega \in \mathfrak{n}_\varphi$.
  Now by Lemma \ref{Lem=OpenLemma},
 \[
 \begin{split}
 (z_j)_\omega D_\varphi^{\frac{1}{2}} = \lim_{\rho \in \mathcal{F}} (z_j)_\omega D_{\rho}^{\frac{1}{2}}
 =  \lim_{\rho \in \mathcal{F}} (z_j)_\omega (D_{\rho_j}^{\frac{1}{2}})_\omega
 =  \lim_{\rho \in \mathcal{F}} ( z_j   D_{\rho_j}^{\frac{1}{2}} )_\omega,
 \end{split}
 \]
 and we claim that this limit equals $( z_j   D_{\varphi_j}^{\frac{1}{2}} )_\omega$. To see this  let $y_j \in \mathfrak{n}_{\varphi_j}$ be bounded in $j$ such that also $\Vert y_j^\ast \Vert_{2, \varphi_j}$ is bounded. Then, using the notation of Lemma \ref{Lem=TakingTheSup} and using \cite[Theorem 5.1]{Raynaud} in the last equality,
\[
\begin{split}
& \lim_{\rho \in \mathcal{F}} ( z_j   D_{\rho_j}^{\frac{1}{2}} )_\omega ( y_j^\ast)_\omega
 =
 \lim_{\rho \in \mathcal{F}} ( z_j   D_{\rho_j}^{\frac{1}{2}} y_j^\ast )_\omega \\
 = &
 \lim_{\rho \in \mathcal{F}} ( z_j   D_{\varphi_j}^{\frac{1}{4}} x_{\rho_j} D_{\varphi_j}^{\frac{1}{4}} y_j^\ast )_\omega
 =
\lim_{\rho \in \mathcal{F}} ( z_j   D_{\varphi_j}^{\frac{1}{4}})_\omega (x_{\rho_j})_\omega (D_{\varphi_j}^{\frac{1}{4}} y_j^\ast )_\omega.
\end{split}
 \]
which converges to $( z_j   D_{\varphi_j}^{\frac{1}{2}}  y_j^\ast )_\omega = ( z_j   D_{\varphi_j}^{\frac{1}{2}})_\omega  (y_j^\ast )_\omega$ by Lemma \ref{Lem=TakingTheSup}.  Then let $\xi$ be the norm limit $\lim_{\rho \in \mathcal{F}} ( z_j   D_{\rho_j}^{\frac{1}{2}} )_\omega$. We see that for all $(y_j)_\omega$ with $\Vert y_j^\ast \Vert_{2, \varphi_j}$ bounded we have $\xi (y_j^\ast)_\omega =    ( z_j   D_{\varphi_j}^{\frac{1}{2}} )_\omega (y_j^\ast )_\omega $. But this can only happen if $\xi =  ( z_j   D_{\varphi_j}^{\frac{1}{2}} )_\omega$ (indeed we can take a net of $(y_j)_\omega$'s that are positive and increasing to 1 as follows from the argument given in the proof of Lemma \ref{Lem=TakingTheSup}).  This concludes the proof.
 \end{proof}

 \begin{lem}\label{Lem=SymmetricUltra}
  Let $\varphi = (\varphi_j)_\omega$ be an ultrapower of normal, semi-finite, faithful weights $\varphi_j$. Assume that $\varphi$ is semifinite.  Let $y_j, z_j \in \mathfrak{n}_{\varphi_j}$ be bounded in $j$ such that also  $\Vert y_j \Vert_{2, \varphi_j}$ and $\Vert z_j \Vert_{2, \varphi_j}$ are bounded. Then,
  \[
  (\varphi_{y_j^\ast z_j}^{(-1/2)})_\omega =   \varphi_{(y_j^\ast z_j)_\omega}^{(-1/2)}
  \]
 \end{lem}
 \begin{proof}
 We have using Lemma \ref{Lem=OpenLemma} for the second equality,
 \[
 \begin{split}
&  \langle  (\varphi_{y_j^\ast z_j}^{(-1/2)})_\omega, (x_j)_\omega   \rangle
= \langle J_\omega (x_j)_\omega^\ast J_\omega (y_j D_{\varphi_j}^{\frac{1}{2}})_\omega, (z_j D_{\varphi_j}^{\frac{1}{2}})_\omega \rangle \\
= &  \langle J_\omega (x_j)_\omega^\ast J_\omega (y_j)_\omega D_{\varphi}^{\frac{1}{2}} , (z_j)_\omega D_{\varphi}^{\frac{1}{2}}  \rangle
=  \langle   \varphi_{(y_j^\ast z_j)_\omega}^{(-1/2)} , (x_j)_\omega   \rangle.
 \end{split}
 \]
  \end{proof}

\begin{lem} \label{Lem=AnotherApproximation}
 Let $\varphi = (\varphi_j)_\omega$ be an ultrapower of normal, semi-finite, faithful weights $\varphi_j$. Assume that $\varphi$ is semifinite.  Let $z \in \nphi$. Then there exists a  net $z_k  = (z_{k, j})_\omega$ with $z_{k, j} \in \mathfrak{n}_{\varphi_j}$ such that  $\Vert z_{k, j} \Vert_{2, \varphi_j}$ is  bounded in $j$ and  $p_\varphi (z_{k, j})_\omega p_\varphi$ is bounded in $k$. Moreover  $p_\varphi (z_{k, j})_\omega p_\varphi \rightarrow  p_\varphi z p_\varphi$ strongly and $\Lambda_\varphi(p_\varphi (z_{k, j})_\omega p_\varphi ) \rightarrow \Lambda_\varphi( p_\varphi z p_\varphi)$ in norm.
\end{lem}
\begin{proof}
To start with take $z \in \mphi^+$.  Consider the set $\mathcal{A}$ of elements $a = (a_j)_\omega$ with $a_j \in \mathfrak{m}_{\varphi_j}^+$ and such that $\Vert \sqrt{a_j} \Vert_{2, \varphi_j}$ is bounded and $\varphi_{(a_j)_\omega}^{(-1/2)} \leq \varphi_z^{(-1/2)}$. The same proof as in Theorem \ref{Thm=NetStructure} shows that $\mathcal{A}$ is a net with respect to the order on positive elements. We claim that
\begin{equation}\label{Eqn=SupSet}
\sup_{a \in \mathcal{A}} p_\varphi a p_\varphi = p_\varphi z p_\varphi.
\end{equation}
Indeed let $y$ be such that $\sup_{a \in \mathcal{A}} p_\varphi a p_\varphi = p_\varphi y p_\varphi$. Note that $p_\varphi y p_\varphi \leq p_\varphi z p_\varphi$ so that $p_\varphi y p_\varphi \in \mphi^+$.   As the weights $\varphi_j$ are normal semi-finite and faithful, the functionals $\varphi_{j, b}^{(-1/2)}$ with $b \in \mathfrak{m}_{\varphi_j}^+$ are dense in $(\cM_j)_\ast^+$. Therefore we may find $b_j \in \mathfrak{m}_{\varphi_j}^+$ such that $\varphi_{z-y}^{(-1/2)} = (  \varphi_{j, b_j}^{(-1/2)}   )_\omega$. We may assume moreover that $\Vert \varphi_{j, b_j}^{(-1/2)}  \Vert \leq \Vert \varphi_{z-y}^{(-1/2)} \Vert$. That is $\varphi_j(b_j) \leq \Vert \varphi_{z-y}^{(-1/2)} \Vert$. Then set $\widetilde{b}_j = b_j(1+b_j)^{-1}$. We get that $\widetilde{b}_j \leq b_j$ and $\Vert \sqrt{\widetilde{b}_j} \Vert_{2, \varphi_j} \leq \varphi_j(b_j)$ is bounded. Set $\widetilde{b} = ( \widetilde{b}_j)_\omega$. In particular this shows using Lemma \ref{Lem=SymmetricUltra} that $\varphi_{\widetilde{b}}^{(-1/2)} = \varphi_{j, (\widetilde{b}_j)_\omega}^{(-1/2)} \leq  (\varphi_{j,  b_j}^{(-1/2)})_\omega  = \varphi_{z-y}^{(-1/2)}$. We conclude that $\widetilde{b} := (\widetilde{b}_j)_\omega \in \mathcal{A}$. This also implies that $p_\varphi \widetilde{b} p_\varphi \leq p_\varphi (z-y) p_\varphi$.   But then for every $a \in \mathcal{A}$ we have
\[
p_\varphi (a+\widetilde{b}) p_\varphi \leq p_\varphi y p_\varphi + p_\varphi (z-y) p_\varphi = p_\varphi z p_\varphi.
\]
We see that for every $a \in \mathcal{A}$ we have $a+\widetilde{b} \in \mathcal{A}$. But this can only happen if $\widetilde{b} = 0$ in which case $p_\varphi (z-y) p_\varphi = 0$. This concludes \eqref{Eqn=SupSet}. As suprema are attained in the strong toplogy we get,
\[
\lim_{a \in \mathcal{A}} p_\varphi a p_\varphi = p_\varphi z p_\varphi.
\]
We also get that,
\begin{equation}\label{Eqn=ToZero}
       \Vert \Lambda_\varphi( p_\varphi z p_\varphi - p_\varphi a p_\varphi  ) \Vert_2^2
= \varphi( (p_\varphi z p_\varphi)^2) +   \varphi( (p_\varphi a p_\varphi)^2) -
\varphi( p_\varphi z p_\varphi a p_\varphi) - \varphi( p_\varphi a p_\varphi z p_\varphi)
\end{equation}
Note that $\sqrt{p_\varphi z p_\varphi} \in \nphi$ and we may write
\[
p_\varphi a p_\varphi = (p_\varphi z p_\varphi)^{1/2}
( (p_\varphi z p_\varphi)^{-1/2}   (p_\varphi a p_\varphi)   (p_\varphi z p_\varphi)^{-1/2} )
(p_\varphi z p_\varphi)^{1/2},
\]
and  $x_a:= (p_\varphi z p_\varphi)^{-1/2}   (p_\varphi a p_\varphi)   (p_\varphi z p_\varphi)^{-1/2}$ is increasing over $a \in \mathcal{A}$ with supremum equal to the support of $p_\varphi z p_\varphi$. By normality of $\varphi$ we see that,
\[
\begin{split}
& \varphi( (p_\varphi a p_\varphi)^2) \rightarrow \varphi( (p_\varphi z p_\varphi)^2), \qquad \textrm{ and }\\
& \varphi( p_\varphi z p_\varphi a p_\varphi) =
 \varphi( (p_\varphi z p_\varphi)^{\frac{3}{2}}   x_a  (p_\varphi z p_\varphi)^{\frac{1}{2}}    ) \rightarrow \varphi( (p_\varphi z p_\varphi)^2).
 \end{split}
\]
Similarly $\varphi( p_\varphi a p_\varphi z p_\varphi) \rightarrow \varphi( (p_\varphi z p_\varphi)^2)$. So \eqref{Eqn=ToZero} converges to 0.

Next we need to prove the lemma for general $z \in \nphi$ through a standard core argument. Firstly note that as every element in $\mphi$ can be written as the sum of four elements in $\mphi^+$, the result follows for $\mphi$.
Now take $z \in p_\varphi \nphi p_\varphi$. Let $\{e_l\}_l$ be a bounded net in $p_\varphi \nphi p_\varphi$ of positive elements approximating 1 in the strong topology. Note that $e_l z \in \mphi$. We have $e_l z \rightarrow z$ strongly and $\Lambda_\varphi( p_\varphi e_l z p_\varphi ) \rightarrow \Lambda(p_\varphi z p_\varphi)$ in norm. In turn the beginning of the proof shows that we may approximate $p_\varphi e_l z p_\varphi$  in the graph norm of $\Lambda_\varphi$ with bounded nets of elements $p_\varphi z_k p_\varphi = p_\varphi (z_{k,j})_\omega p_\varphi$ such that $\Vert z_{k,j} \Vert_{2, \varphi_j}$ is bounded in $j$ and $p_\varphi (z_{k,j})_\omega p_\varphi$ is bounded in $k$.
\end{proof}

\begin{lem}\label{Lem=SubtleLemma}
 Let $\varphi = (\varphi_j)_\omega$ be an ultrapower of normal, semi-finite, faithful weights $\varphi_j$. Assume that $\varphi$ is semifinite.   Let $p_\varphi$ be the support of $\varphi$. Then there exists a bounded net $z_k = (z_{k, j})_\omega$ on $\cM$ such that $z_k \rightarrow p_\varphi$ in the $\sigma$-weak topology and such that $z_{k,j}$ are analytic for $\sigma^{\varphi_j}$ with  $\sigma_{i/2}^{\varphi_j}(z_{k, j})$   bounded (in both $j$ and $k$) and with $(\sigma_{i/2}^{\varphi_j}(z_{k, j}) )_\omega \rightarrow p_\varphi$ in the $\sigma$-weak topology.
\end{lem}
\begin{proof}
Firstly for $t \in \mathbb{R}$ consider the ultraproduct of $\ast$-homomorphisms $(\sigma_t^{\varphi_j})_\omega$. For $x \in \cM$ we have $\langle  (\rho_j)_\omega, (\sigma_t^{\varphi_j})_\omega(x) \rangle = \langle   (\rho_j \circ \sigma_t^{\varphi_j})_\omega, x \rangle$. Note that $t \mapsto (\sigma_t^{\varphi_j})_\omega$ is not necessarily strongly continuous (in fact in \cite{AndoHaagerup} it is proved it is usually not) but it is measurable in the sense that for $f \in L^1(\mathbb{R})$ we define $\int_\mathbb{R} f(t) (\sigma_t^{\varphi_j})_\omega(x) dt$ as the element determined by $\langle  (\rho_j)_\omega, \int_\mathbb{R} f(t) (\sigma_t^{\varphi_j})_\omega(x) dt \rangle = \langle  ( \int_{\mathbb{R}} f(t) \rho_j \circ \sigma^{\varphi_j}_t)_\omega, x \rangle$.
 
 Now we claim that $p_\varphi$ is invariant for this homomorphism. Indeed, take $x = (x_j)_\omega \in \cM$ with $\Vert x_j \Vert$ and $\Vert x_j^\ast \Vert_{2, \varphi}$ bounded.  Then, applying Lemma \ref{Lem=OpenLemma} twice,
\[
(\Delta_{j}^{it})_\omega (D_{\varphi}^{\frac{1}{2}} x) =
(\Delta_{j}^{it})_\omega (D_{\varphi_j}^{\frac{1}{2}} x_j)_\omega=
(D_{\varphi_j}^{\frac{1}{2}} \sigma^{\varphi_j}_t( x_j))_\omega =
D_{\varphi}^{\frac{1}{2}} (\sigma^{\varphi_j}_t( x_j))_\omega.
\]
Hence $(\Delta^{it}_j)_\omega (D_\varphi^{\frac{1}{2}} x ) \subseteq p_\varphi \cH$. Since $((\Delta^{it}_j)_\omega)^{-1} = (\Delta^{-it}_j)_\omega$ we see that actually $(\Delta^{it}_j)_\omega p_\varphi \cH = p_\varphi \cH$. Then it follows that $(\sigma^{\varphi_j}_t)_\omega(p_\varphi) = (\Delta^{it}_j)_\omega p_\varphi (\Delta^{-it}_j)_\omega = p_\varphi$.

By Kaplansky's density theorem let $y_k  = (y_{k, j})_\omega$ be a bounded net converging in the $\sigma$-weak topology to $p_\varphi$. We set $z_{k,j} = \frac{1}{\sqrt{\pi}} \int_{-\infty}^{\infty} e^{-t^2} \sigma_t^{\varphi_j}(y_{k,j}) dt$. Then we have $\sigma^{\varphi_j}_v( z_{k,j} ) = \frac{1}{\sqrt{\pi}} \int_{-\infty}^{\infty} e^{-(t-v)^2} \sigma_t^{\varphi_j}(y_{k,j}) dt$ which is bounded. Also, for $\rho = (\rho_j)_\omega \in \cM_\ast$,
\[
\begin{split}
 \langle \rho, (z_{k,j})_\omega \rangle = &
\langle (  \frac{1}{\sqrt{\pi}} \int_{-\infty}^{\infty} e^{-t^2} \rho_j \circ \sigma_t^{\varphi_j}( \: \cdot \:) dt)_\omega, (y_{k,j})_\omega \rangle \\
\rightarrow &
\langle (  \frac{1}{\sqrt{\pi}} \int_{-\infty}^{\infty} e^{-t^2} \rho_j \circ \sigma_t^{\varphi_j}( \: \cdot \:) dt)_\omega, p_\varphi \rangle \\
= & \langle  (\rho_j)_\omega, \int_{-\infty}^\infty e^{-t^2} (\sigma^{\varphi_j}_t)_\omega(p_\varphi) dt \rangle \\
= & \langle  \rho, p_\varphi \rangle,
\end{split}
\]
so that $ (z_{k,j})_\omega \rightarrow p_\varphi$ $\sigma$-weakly. Similarly,
\[
\begin{split}
\langle \rho, ( \sigma^{\varphi_j}_{i/2} (z_{k,j})   )_\omega \rangle = &
\langle (  \frac{1}{\sqrt{\pi}} \int_{-\infty}^{\infty} e^{-(t-i/2)^2} \rho_j \circ \sigma_t^{\varphi_j}( \: \cdot \:) dt)_\omega, (y_{k,j})_\omega \rangle \\
\rightarrow &
\langle (  \frac{1}{\sqrt{\pi}} \int_{-\infty}^{\infty} e^{-(t-i/2)^2} \rho_j \circ \sigma_t^{\varphi_j}( \: \cdot \:) dt)_\omega, p_\varphi \rangle = \langle \rho, p_\varphi \rangle.
\end{split}
\]
This concludes the proof.
\end{proof}

\begin{lem}\label{Lem=StandardTomTak}
Let $\varphi$ be a normal, semi-finite, faithful weight on a von Neumann algebra $\cN$. Let $x \in \nphi$. Let $a \in\cN$ be analytic for $\sigma^{\varphi}$. Then $x a \in \nphi$ and $\Vert x a \Vert_{2, \varphi} \leq \Vert \sigma^{\varphi}_{i/2}( a) \Vert \Vert x \Vert_{2, \varphi}$.
\end{lem}
\begin{proof}
That $xa \in \nphi$ follows from \cite[Lemma VIII.2.4]{TakII}. We have
\[
\begin{split}
& \langle \Lambda_{\varphi}(xa), \Lambda_{\varphi}(x a) \rangle =
\langle J \sigma^{\varphi}_{-i/2}(a^\ast) J \Lambda_{\varphi}(x),  J \sigma^{\varphi}_{-i/2}(a^\ast) J \Lambda_{\varphi}(x) \rangle \\
= & \langle J \sigma^{\varphi}_{i/2}(a) \sigma^{\varphi}_{i/2}(a)^\ast J \Lambda_{\varphi}(x), \Lambda_{\varphi}(x) \rangle \leq \Vert \sigma^{\varphi}_{i/2}( a) \Vert^2 \Vert x \Vert_{2, \varphi}^2.
\end{split}
\]
\end{proof}

\begin{thm}\label{Thm=OneMoreSpatial}
Let $\varphi_j$ be normal, semi-finite, faithful weights on $\cM_j$ with ultraproduct $\varphi = (\varphi_j)_\omega$. Let $\psi_j$ be normal, semi-finite, faithful weights on $\cM_j'$ with ultraproduct $\psi = (\psi_j)_\omega$. Let $p_\varphi$ be the support of $\varphi$ and let $r_\psi$ be the support of $\psi$. Put $p_\psi = J_\omega r_\psi J_\omega$ and $q_\psi = p_\psi r_\psi$. Assume that both $\varphi$ and $\psi$ are semi-finite and $p_\varphi \leq p_\psi$. Then we have,
\[
q_\psi \left( \frac{d\varphi_j}{ d \psi_j}  \right)_\omega q_\psi = \frac{d \varphi}{ d\psi}.
\]
Furthermore $q_\psi$ and $\left( \frac{d\varphi_j}{ d \psi_j}  \right)_\omega$ commute.
\end{thm}
\begin{proof}
Let $\mathcal{A}$ be the set of all vectors $a = (a_j)_\omega$   such that $\Vert a_j \Vert_{2, \varphi_j}$ is bounded. Let $\mathcal{B}$ be the set of all vectors $b = (b_j)_\omega$  such that $\Vert b_j  \Vert_{2, \psi_j}$ is bounded.   Firstly we claim that the linear span of the vectors $p_\psi \mathcal{A}^\ast p_\psi \mathcal{B} p_\psi D_{\psi}^{\frac{1}{2}}$   forms a core for $\left( \frac{d \varphi}{d\psi} \right)^{\frac{1}{2}}$. Indeed by Lemma \ref{Lem=CoreOfSpatialDerivative} it suffices to show that every element   $y \in   p_\psi \nphi^\ast  p_\psi \npsi p_\psi$, say $y = p_\psi y_1^\ast p_\psi y_2 p_\psi$ with $y_1 \in \nphi, y_2 \in \npsi$, may be approximated by elements in ${\rm span } \: p_\psi \mathcal{A}^\ast p_\psi \mathcal{B}  p_\psi D_{\psi}^{\frac{1}{2}}$  in the graph norm of $\left( \frac{d \varphi}{d\psi} \right)^{\frac{1}{2}}$. Now it follows from Lemma \ref{Lem=AnotherApproximation} that we may find   nets $\{a_k\}_k$ in $\mathcal{A}$ and $\{b_l\}_l$ in  $\mathcal{B}$ such that $p_\psi a_k p_\psi$ and $p_\psi b_l p_\psi$ are bounded, such that $a_k^\ast \rightarrow y_1^\ast$ and $b_k \rightarrow y_2$  strongly and such that $D_{\varphi}^{\frac{1}{2}} p_\psi a_k^\ast p_\psi \rightarrow D_{
\varphi}^{\frac{1}{2}} p_\psi y_1^\ast p_\psi$ and $p_\psi b_l p_\psi D_{\psi}^{\frac{1}{2}}  \rightarrow p_\psi y_2 p_\psi D_{\psi}^{\frac{1}{2}}$ in norm.
Then,
\[
\begin{split}
& \Vert p_\psi a_k^\ast p_\psi b_l p_\psi D_\psi^{\frac{1}{2}} - p_\psi y_1^\ast p_\psi y_2 p_\psi D_\psi^{\frac{1}{2}} \Vert_2\\
  \leq &
  \Vert p_\psi a_k^\ast p_\psi b_l  p_\psi D_\psi^{\frac{1}{2}} - p_\psi a_k^\ast p_\psi y_2 p_\psi D_\psi^{\frac{1}{2}} \Vert_2 +
\Vert p_\psi a_k^\ast  p_\psi y_2  p_\psi D_\psi^{\frac{1}{2}} - p_\psi y_1^\ast  p_\psi y_2 p_\psi D_\psi^{\frac{1}{2}}  \Vert_2 \\
 \leq  &   \Vert p_\psi a_k^\ast p_\psi \Vert \Vert p_\psi b_l p_\psi D_\psi^{\frac{1}{2}} - p_\psi y_2 p_\psi D_\psi^{\frac{1}{2}} \Vert_2 +
\Vert p_\psi a_k^\ast p_\psi y_2 p_\psi D_\psi^{\frac{1}{2}} - p_\psi y_1^\ast p_\psi y_2 p_\psi D_\psi^{\frac{1}{2}}  \Vert_2,
\end{split}
\]
which goes to 0. Similarly $\Vert D_{\varphi}^{\frac{1}{2}} p_\psi a_k^\ast p_\psi b_l  p_\psi  - D_\varphi^{\frac{1}{2}}p_\psi y_1^\ast p_\psi y_2 p_\psi \Vert_2 \rightarrow 0$. This concludes the claim, see Theorem \ref{Thm=Spacial}.

Now fix $a = (a_j)_\omega \in \mathcal{A}$ and $b = (b_j)_\omega \in \mathcal{B}$. Let $z_k = (z_{k,j})_\omega$ be a bounded net in $\cM$ converging to $p_\psi$ strongly.  Assume moreover that each $z_{k,j}$ is analytic for $\sigma^{\varphi_j}$ and that $\sigma^{\varphi_j}_{i/2}(z_{k,j}^\ast)$ is bounded uniformly in $j$ and that $(\sigma^{\varphi_j}_{i/2}(z_{k,j}^\ast))_\omega \rightarrow p_\psi$ strongly, see Lemma \ref{Lem=SubtleLemma}. This ensures that $(a_j z_{k,j}^\ast)_\omega \in \mathcal{A}$ by Lemma \ref{Lem=StandardTomTak}.
Set $x_{k,l} = z_k a^\ast z_l b$ so that $x_{k,l} = (x_{k,l,j})_\omega = (z_{k,j} a_j^\ast z_{l,j} b_j)_\omega$.  We get by Theorem \ref{Thm=Spacial}, Lemma \ref{Lem=OpenLemma} and again Theorem \ref{Thm=Spacial},
\begin{equation}\label{Eqn=House}
\begin{split}
& \left( \frac{d \varphi}{d\psi} \right)^{\frac{1}{2}} p_\psi x_{k,l} p_\psi D_\psi^{\frac{1}{2}} = D_\varphi^{\frac{1}{2}} p_\psi x_{k,l} p_\psi = D_\varphi^{\frac{1}{2}} x_{k,l} p_\psi \\
= & (D_{\varphi_j}^{\frac{1}{2}} x_{k,l,j})_\omega p_\psi = \left(  \left( \frac{d \varphi_j}{d \psi_j} \right)^{\frac{1}{2}} x_{k,l,j} D_{\psi_j}^{\frac{1}{2}} \right)_\omega p_\psi.
\end{split}
\end{equation}
So that, by Lemma \ref{Lem=UltraOfElements} and Lemma \ref{Lem=OpenLemma},
\[
\begin{split}
&  \Vert \left( \frac{d \varphi}{d \psi} \right)^{\frac{1}{2}} p_\psi x_{k,l} p_\psi D_\psi^{\frac{1}{2}} \Vert_2^2  \geq
\langle  \left( \frac{d \varphi_j}{d \psi_j} \right)_\omega, \rho_{  (x_{k,l,j} D_{\psi_j}^{\frac{1}{2}})_\omega,  (x_{k,l,j} D_{\psi_j}^{\frac{1}{2}})_\omega } \rangle  =
 \langle  \left( \frac{d \varphi_j}{d \psi_j} \right)_\omega, \rho_{  x_{k,l} D_{\psi}^{\frac{1}{2}},  x_{k,l} D_{\psi}^{\frac{1}{2}}} \rangle.
\end{split}
\]
Taking the limit in $k,l$ gives $x_{k,l} p_\psi D_{\psi}^{\frac{1}{2}} \rightarrow p_\psi a^\ast p_\psi b p_\psi D_\psi^{\frac{1}{2}}$ and $D_{\varphi}^{\frac{1}{2}} p_\psi x_{k,l} p_\psi  \rightarrow D_{\varphi}^{\frac{1}{2}} p_\psi a^\ast p_\psi b p_\psi$ in norm. This shows that (using closedness of the quadratic form associated with $\left( \frac{d \varphi_j}{d \psi_j} \right)_\omega$),
\[
\begin{split}
& \Vert \left( \frac{d \varphi}{d \psi} \right)^{\frac{1}{2}} p_\psi a^\ast p_\psi b p_\psi D_\psi^{\frac{1}{2}} \Vert_2^2
= \lim_{k,l} \Vert \left( \frac{d\varphi}{d \psi} \right)^{\frac{1}{2}} p_\psi x_{k,l} p_\psi D_{\psi}^{\frac{1}{2}} \Vert_2^2 \\
\geq & \lim_{k,l} \langle \left( \frac{d \varphi_j}{d \psi_j} \right)_\omega, \rho_{x_{k,l} D_{\psi}^{\frac{1}{2}}, x_{k,l} D_{\psi}^{\frac{1}{2}} } \rangle
= \langle  \left( \frac{d \varphi_j}{d \psi_j} \right)_\omega,  \rho_{p_\psi a^\ast p_\psi b p_\psi D_\psi^{\frac{1}{2}}, p_\psi a^\ast p_\psi b p_\psi D_\psi^{\frac{1}{2}}} \rangle.
\end{split}
\]
By Lemma \ref{Lem=AnotherApproximation} the span of the set of vectors $p_\psi a^\ast p_\psi b p_\psi D_\psi^{\frac{1}{2}}$ with $a \in \mathcal{A}$ and $b \in \mathcal{B}$ is dense in $p_\psi \cH p_\psi = q_\psi(\cH)$. We see that $q_\psi \leq P_{< \infty}\left(  \left( \frac{d \varphi_j}{d \psi_j} \right)_\omega \right)$. Therefore, from \eqref{Eqn=House} and Lemma \ref{Lem=UltraOfElements},
\[
\begin{split}
 & \left( \frac{d \varphi}{d \psi} \right)^{\frac{1}{2}} p_\psi x_{k,l} p_\psi D_\psi^{\frac{1}{2}}
 =  p_\psi \left( \frac{d \varphi}{d \psi} \right)^{\frac{1}{2}} p_\psi x_{k,l} p_\psi D_\psi^{\frac{1}{2}}   p_\psi \\
= & p_\psi \left( \left( \frac{d \varphi_j}{ d \psi_j} \right)^{\frac{1}{2}} x_{k,l,j} D_{\psi_j}^{\frac{1}{2}}  \right)_\omega p_\psi
= q_\psi \left( \frac{d \varphi_j}{d \psi_j} \right)_\omega^{\frac{1}{2}} (x_{k,l,j} D_{\psi_j}^{\frac{1}{2}})_\omega.
\end{split}
\]
and taking limits in $k,l$ gives (using closedness of the operators),
\begin{equation}\label{Eqn=Coreish}
\left( \frac{d \varphi}{d \psi} \right)^{\frac{1}{2}} p_\psi a^\ast p_\psi b p_\psi D_{\psi}^{\frac{1}{2}} =
= q_\psi \left( \frac{d \varphi_j}{d \psi_j} \right)_\omega^{\frac{1}{2}}   ( p_\psi a^\ast p_\psi b p_\psi D_{\psi}^{\frac{1}{2}})
= q_\psi \left( \frac{d \varphi_j}{d \psi_j} \right)_\omega^{\frac{1}{2}} q_\psi ( p_\psi a^\ast p_\psi b p_\psi D_{\psi}^{\frac{1}{2}}).
\end{equation}
We proved in the first paragraph  that the span of the vectors $p_\psi a^\ast p_\psi b p_\psi D_{\psi}^{\frac{1}{2}}$ with $a \in \mathcal{A}$ and $b \in \mathcal{B}$ forms a core for  $\left( \frac{d \varphi}{d \psi} \right)^{\frac{1}{2}}$ so that we get from   \eqref{Eqn=Coreish} that $\left( \frac{d \varphi}{d \psi} \right)^{\frac{1}{2}} \subseteq q_\psi \left(  \frac{d \varphi_j}{d \psi_j} \right)_\omega^{\frac{1}{2}} q_\psi$ and as both sides are self-adjoint we actually have
\[
\left( \frac{d \varphi}{d \psi} \right)^{\frac{1}{2}} = q_\psi \left(  \frac{d \varphi_j}{d \psi_j} \right)_\omega^{\frac{1}{2}} q_\psi.
\]
It remains to prove that $q_\psi(\cH)$ is an invariant subspace of $\left( \frac{d \varphi_j}{d \psi_j} \right)_\omega$. In order to do so take a bounded net $z_k = (z_{k,j})$ such that $z_k \rightarrow p_\psi$ strongly and such that each $z_{k,j}$ is analytic for $\sigma^{\varphi_j}$ with $(\sigma^{\varphi_j}_{-i/2}(z_{k,j}))_\omega$ bounded and strongly convergent to $p_\psi$, c.f. Lemma \ref{Lem=SubtleLemma}. Set $P_{< \infty} = P_{< \infty}\left(  \left( \frac{d \varphi_j}{d \psi_j} \right)_\omega \right)$. Then, using closedness of the operators involved, Lemma \ref{Lem=UltraOfElements} for the third equality, and using that we saw that $q_\psi \leq P_{<\infty}$,
\[
\begin{split}
& \left( \frac{d \varphi_j}{d \psi_j} \right)_\omega^{\frac{1}{2}}  p_\psi a^\ast p_\psi b p_\psi D_{\psi}^{\frac{1}{2}}
= \lim_{k,l,m} \left( \frac{d \varphi_j}{d \psi_j} \right)_\omega^{\frac{1}{2}}  z_k a^\ast z_l b z_m D_{\psi}^{\frac{1}{2}} \\
= & \lim_{k,l,m} \left( \frac{d \varphi_j}{d \psi_j} \right)_\omega^{\frac{1}{2}}  (z_{k,j} a_j^\ast z_{l,j} b_j z_{m,j} D_{\psi_j}^{\frac{1}{2}})_\omega
= \lim_{k,l,m} P_{< \infty} \left(   ( D_{\varphi_j}^{\frac{1}{2}} z_{k,j} a_j^\ast z_{l,j} b_j z_{m,j} \right)_\omega \\
= & P_{< \infty}  ( \sigma_{-i/2}^{\varphi_j}(z_{k,j}) )_\omega ( D_{\varphi_j}^{\frac{1}{2}} a_j^\ast )_\omega z_l (b_j)_\omega z_m
 =  P_{< \infty}  p_\psi (D_{\varphi_j}^{\frac{1}{2}} a_j^\ast)_\omega p_\psi (b_j)_\omega p_\psi \\
= &   p_\psi (D_{\varphi_j}^{\frac{1}{2}} a_j^\ast) p_\psi (b_j)_\omega p_\psi  \in q_\psi(\cH).
\end{split}
\]
So we conclude that $\left( \frac{d \varphi_j}{d \psi_j} \right)_\omega^{\frac{1}{2}}$ preserves $q_\psi(\cH)$.

\end{proof}

\subsection{Modular theory} In this section we show that Theorem \ref{Thm=OneMoreSpatial} implies  results on modular theory and Radon-Nikodym derivatives. We first prove Theorem A of the introduction.

\begin{thm}\label{Thm=Modular}
Let $\varphi_j$ be  normal, semi-finite, faithful weights on $\cM_j$. Let $\varphi = (\varphi_j)_\omega$ be the ultraproduct weight on $\cM$ with support $p_\varphi$. Assume that $\varphi$ is semi-finite. 
Then,
\[
\sigma^{\varphi}_t( p_\varphi (x_j)_\omega p_\varphi ) =  p_\varphi (\sigma_t^{\varphi_j}(x_j))_\omega p_\varphi.
\]
\end{thm}
\begin{proof}
Let $\psi_j$ be the opposite weight of $\varphi_j$ (see \cite{TakII}) and let $\psi = (\psi_j)_\omega$ be its ultraproduct.
Letting again $r_{\psi_j}$ be the support of $\psi_j$ and $p_{\psi_j} = J_j r_{\psi_j} J_j$. We have $p_{\varphi_j} = p_{\psi_j}$ and as the support of $\psi$ equals $J_\omega p_\varphi J_\omega$ we also have $p_\varphi = p_\psi$. Hence we may apply Theorem \ref{Thm=OneMoreSpatial} to find that,
\begin{equation}\label{Eqn=SpatialEquality}
\frac{d \varphi}{ d\psi} = q_\psi \left(  \frac{d \varphi_j}{ d \psi_j} \right)_\omega q_\psi.
\end{equation}
Now since $p_\varphi = p_\psi$ the kernel of $\frac{d \varphi}{ d\psi}$ acting on $q_\psi(\cH)$ is 0. Therefore \eqref{Eqn=SpatialEquality} shows that the restriction of  $\left( \frac{d \varphi_j}{ d \psi_j} \right)_\omega$ to the invariant subspace $ q_\psi(\cH)$ has kernel equal to 0 and $q_\psi$ is disjoint from the infinite projection in the spectral resolution of $\left( \frac{d \varphi_j}{ d \psi_j} \right)_\omega$ (see \eqref{Eqn=Integral}). This shows that $q_\psi(\cH)$ is contained in $\cup_{\lambda >0} \chi_{[\frac{1}{\lambda}, \lambda]}(  \left( \frac{d \varphi_j}{d \psi_j} \right)_\omega ) \cH$. Applying Theorem \ref{Thm=OneMoreSpatial} and Theorem \ref{Thm=SpectralStuff} and noting in addition that $ q_\psi$ and $\left( \left(  \frac{d \varphi_j}{ d \psi_j} \right)^{it}\right)_\omega$ commute then shows,
\[
\left( \frac{d \varphi}{ d\psi} \right)^{it} =
\left(  q_\psi \left(  \frac{d \varphi_j}{ d \psi_j} \right)_\omega  q_\psi \right)^{it}
=  q_\psi \left( \left(  \frac{d \varphi_j}{ d \psi_j} \right)^{it}\right)_\omega  q_\psi.
\]
Then we apply \eqref{Eqn=SpacialModular} twice to get,
\[
\begin{split}
  &\sigma^{\varphi}_t(  q_\psi (x_j)_\omega  q_\psi)
=  \left( \frac{d \varphi}{ d\psi} \right)^{it}  q_\psi (x_j)_\omega  q_\psi \left( \frac{d \varphi}{ d\psi} \right)^{-it}\\
= &   q_\psi \left(  \left(\frac{d \varphi_j}{ d\psi_j}\right)^{it}\right)_\omega    q_\psi (x_j)_\omega q_\varphi \left(  \left(\frac{d \varphi_j}{ d\psi_j}  \right)^{-it}\right)_\omega  q_\psi
=     q_\psi \left(  \left(\frac{d \varphi_j}{ d\psi_j}\right)^{it}\right)_\omega     (x_j)_\omega  \left(  \left(\frac{d \varphi_j}{ d\psi_j}  \right)^{-it}\right)_\omega  q_\psi\\
= &   q_\psi \left(  \left(\frac{d \varphi_j}{ d\psi_j}\right)^{it} x_j   \left(\frac{d \varphi_j}{ d\psi_j}  \right)^{-it}\right)_\omega  q_\psi
=  q_\psi (\sigma_t^{\varphi_j}( x_j )  )_\omega  q_\psi.
\end{split}
\]
This also yields that $\sigma^{\varphi}_t( p_\varphi (x_j)_\omega p_\varphi ) =  p_\varphi (\sigma_t^{\varphi_j}(x_j))_\omega p_\varphi$ as $p_\varphi \cM p_\varphi \rightarrow q_\varphi \cM q_\varphi: x \mapsto q_\varphi x q_\varphi$ is a (non-spatial) isomorphism (see \cite[Lemma 3.26]{AndoHaagerup} or \cite[Corollary 2.5, Lemma 2.6]{HaagerupScan}).
\end{proof}

We refer to \cite[Section VIII.3]{TakII} for the definition of the cocycle derivative.

\begin{thm}\label{Thm=CocycleDerivative}
Let $\varphi_j$ and $\rho_j$ be  normal, semi-finite, faithful weights on $\cM_j$. Let $\varphi = (\varphi_j)_\omega$ and $\rho = (\rho_j)_\omega$ be their ultraproduct weights on $\cM$. Assume that $\rho$ and $\varphi$ are semi-finite with equal support $p$.
Then,
\begin{equation}\label{Eqn=RadonNikodym}
p  \left( \left( \frac{D \varphi_j}{D \psi_j} \right)_t \right)_\omega =
 \left( \left( \frac{D \varphi_j}{D \psi_j} \right)_t \right)_\omega p  =
  \left( \frac{D \varphi}{D \psi} \right)_t.
\end{equation}
\end{thm}
\begin{proof}
Let $\psi = (\psi_j)_\omega$ be an ultraproduct of normal, semi-finite, faithful weights $\psi_j$ on $\cM_j$ such that $\psi$ is semi-finite with support $r = J_\omega p J_\omega$. Set $q = pr$. Then by \cite[Theorem IX.3.8]{TakII}, Theorem \ref{Thm=OneMoreSpatial}, Theorem \ref{Thm=SpectralStuff} (which is applicable, c.f. the proof of Theorem \ref{Thm=Modular}) and once more  \cite[Theorem IX.3.8]{TakII},  we see that,
\[
\begin{split}
& \left(  \frac{D \varphi}{ D \rho}  \right)_t = \left( \frac{d \varphi}{ d \psi} \right)^{it} \left( \frac{d \psi}{ d \rho}  \right)^{it}
= q \left( \frac{d \varphi_j}{d \psi_j} \right)^{it}_\omega q \left( \frac{d \psi_j}{d \rho_j} \right)^{it}_\omega q \\
= & q \left( \frac{d \varphi_j}{d \psi_j} \right)^{it}_\omega   \left( \frac{d \psi_j}{d \rho_j} \right)^{it}_\omega
= q \left( \left( \frac{d \varphi_j}{d \psi_j} \right)^{it}  \left( \frac{d \psi_j}{d \rho_j} \right)^{it}\right)_\omega
=   q \left(  \left( \frac{D \varphi_j}{D \rho_j} \right)_t  \right)_\omega.
\end{split}
\]
Adapting the second line of this equation  in the obvious way also gives
\[
\left(  \frac{D \varphi}{ D \rho}  \right)_t =    \left(  \left( \frac{D \varphi_j}{D \rho_j} \right)_t  \right)_\omega q.
 \]
 Under the isomorphism $p \cM p \simeq q \cM q: x \mapsto q x q$ (c.f. \cite[Lemma 3.26]{AndoHaagerup} or \cite[Corollary 2.5, Lemma 2.6]{HaagerupScan}),  this yields \eqref{Eqn=RadonNikodym}.
\end{proof}

\subsection{Related results}
We end this section with a discussion about the following question. If $\rho \in \cM_\ast$ then it can always be written in the form $\rho = (\rho_j)_\omega$ as $\cM_\ast \simeq \prod_{j, \omega} (\cM_j)_\ast$. But can a weight on $\cM$ always be decomposed in this way? We answer this question in the negative.

\begin{prop}\label{Prop=Decomposition}
There exists an ultraproduct $\cM = \prod_{j, \omega} \cM_j$ with normal, semi-finite, faithful weight $\varphi$ on $\cM$ such that $\varphi$ cannot be written as the ultraproduct of normal, semi-finite, faithful weights on each of the $\cM_j$'s.
\end{prop}
\begin{proof}
Suppose that there exists $x \in \cM^+$ that is not of the form $(x_j)_\omega$, i.e. it is not contained in the Banach space ultraproduct of the $\cM_j$'s (see \cite[Remark 3.6]{AndoHaagerup} for an example). Suppose that $\varphi$ is a normal, semi-finite, faithful weight on $\cM$ such that $x \in \mphi^+$ (note that such a weight always exists), then $\varphi$ cannot be written as an ultraproduct $(\varphi_j)_\omega$ of normal, semi-finite and faithful weights $\varphi_j$ on $\cM_j$. Indeed as $x \in \mphi^+$ we have $\varphi_x^{(-1/2)} \in \cM_\ast^+$ which we write as $\varphi_x^{(-1/2)} = (\varphi_{j, x_j}^{(-1/2)} )_\omega$ with $x_j \in \mathfrak{m}_{\varphi_j}^+$. Let $G_\lambda$ be a continuous function such that $0 \leq G_\lambda \leq 1$ with $G_\lambda(s) =1$ for $s \in [0, \lambda]$ and $G_\lambda(s)  = 0$ for $s \in [2\lambda, \infty)$.  Because by Lemma \ref{Lem=SymmetricUltra} we have  $\varphi^{(-1/2)}_{ G_\lambda(x_j)_\omega  } = (\varphi^{(-1/2)}_{j, G_\lambda(x_j)  })_\omega \leq (  \varphi^{(-1/2)}_{j, x_j}  )_\omega = \varphi^{(-1/2)}_{x}$ and $\varphi$ is faithful it follows that  $G_\lambda(x_j)_\omega \leq x$ and for $\lambda > \Vert x \Vert$ we get equality $(G_\lambda(x_j))_\omega = x$. This contradicts that $x$ cannot be written as an element  $(x_j)_\omega$.
\end{proof}

The following proposition shows that sometimes ultraproduct weights are faithful. Note that we do not claim that $\tau$ in the next proposition is semi-finite, in fact this is not true as for suitable ultrafilters it would contradict \cite[Proposition 1.14]{Raynaud}.

\begin{prop}\label{Prop=Faithful}
Let $\cM_j$ be type I von Neumann algebras. Let $\tau_j: \cM_j \rightarrow [0, \infty]$ be the trace that is determined by $\tau_j(p) = 1$ for every atomic projection $p \in \cM_j$. Then $\tau := (\tau_j)_\omega$ is faithful on $\prod_{j, \omega} \cM_j$.
\end{prop}
\begin{proof}
We first claim that for elementary elements $0 \leq x := (x_j)_\omega \in \prod_{j, \omega} \cM_j$ we have $\tau(x) = 0$ implies $x = 0$. Indeed, we may assume that $x_j \geq 0$.
 If $\tau(x)=0$ then this implies that for every series of atomic projections $(p_j)_\omega$ we have $\langle (p_j \tau_j p_j)_\omega, (x_j)_\omega\rangle = 0$. But this implies that $\lim_{j, \omega} p_j x_j p_j = 0$. So that $(p_j)_\omega (x_j)_\omega (p_j)_\omega = 0$. Let $\lambda_j$ be the largest eigenvalue of $x_j$ and let $p_j'$ be the corresponding spectral projection. Let $p_j$ be an atomic subprojection of $p_j'$. Then $0 = (p_j)_\omega (x_j)_\omega (p_j)_\omega  = (p_j x_j p_j)_\omega = (\lambda_j p_j)_\omega$. So $\lim_{j, \omega} \lambda_j =0$ and therefore $(x_j)_\omega = 0$.

 Now let $0 \leq x \in \prod_{j, \omega} \cM_j$ be arbitrary such that $\tau(x) = 0$. Suppose that $x \not = 0$. Let $(p_j)_\omega$ be a series of atomic projections such that $x(p_j)_\omega$ is non-zero. Then also  $(p_j)_\omega x(p_j)_\omega$ is non-zero and moreover this element is contained in $(p_j)_\omega \left(  \prod_{j, \omega} \cM_j \right) (p_j)_\omega \simeq  \prod_{j, \omega} p_j \cM_j p_j$, so it falls in the case of the first paragraph. Then $0 = \tau(x) \geq \tau( (p_j)_\omega x(p_j)_\omega) \not = 0$, which is a contradiction. Hence $x = 0$.
 \end{proof}

\begin{rmk}
In view of Theorem \ref{Thm=Lp} it would be interesting to know if for every ultraproduct $\cM = \prod_{j, \omega} \cM_j$ we can find some normal, semi-finite, faithful weights $\varphi_j$ on $\cM_j$ such that the ultraproduct $(\varphi_j)_\omega$ is again normal, semi-finite and faithful. Taking into account the previous propositions we do not expect that this holds.
\end{rmk}

\section{Applications to noncommutative $L^p$-spaces}\label{Sect=Lp}

In this section we prove Theorem C of the introduction. 
We  need the following version of the Powers--St\o{}rmer inequality. A proof can be found in \cite[Lemma 2.2]{CPPR}. Note that this lemma only holds for semi-finite von Neumann algebras. The general case can be deduced using \cite[Proposition 7 and Lemma B]{KosAbs}.

\begin{prop}[Powers--St\o{}rmer inequality]\label{Prop=PowersStormer}
Let $\cN$ be an aribtrary (not necessarily semi-finite) von Neumann algebra. We have for $x,y \in L^1(\cN)^+$ that
\[
\Vert x^{\frac{1}{p}} - y^{\frac{1}{p}} \Vert_p^p \leq \Vert x - y \Vert_1.
\]
\end{prop}

Let $\psi$ be a weight on the commutant $\cN'$ of a von Neumann algebra $\cN$ which we assume is in standard form. Recall that the opposite weight is defined as the weight $\psiop$ on $\cN$ that is given by $\psiop(x^\ast x) = \psi(z^\ast z)$ where $z = J x^\ast J$.

\begin{thm}\label{Thm=Lp}
 Let $\psi_j$ be normal, semi-finite, faithful weights on $\cM_j'$ with ultraproduct $\psi = (\psi_j)_\omega$ which we assume to be semi-finite. Let $r_\psi$ be the support of $\psi$ and let $p_\psi = J_\omega r_\psi J_\omega$ and $q_\psi = p_\psi r_\psi$.   We have an isomorphism,
\[
 p_\psi  \left( \prod_{j, \omega} L^p(\cM_j, \psi_j) \right) p_\psi\quad \rightarrow\!\!\!\!\!\!\!^\simeq  \quad L^p( q_\psi  \cM q_\psi, \psi ),
\]
that is determined by
\[
\left(\frac{d \rho_j}{d \psi_j} \right)_\omega^{\frac{1}{p}} \mapsto  \left( \frac{d \rho}{d \psi} \right)^{\frac{1}{p}},
\]
where $\rho = (\rho_j)_\omega$ with $\rho_j \in (\cM_j)_\ast$ uniformly bounded in $j$ and such that $p_\rho \leq p_\psi$. 
\end{thm}
\begin{proof}
Throughout this proof we shall refer to the following special types of elements.
\begin{enumerate}
\item[($\ast$)] Let $A_j \in L^p(\cM_j)^+$ be a  sequence bounded in $j$ in the $\Vert \: \cdot \: \Vert_p$-norm and of the form
\[
A_j = \left(   \frac{d \rho_j}{ d \psi_j}\right)^{\frac{1}{p}}. 
\]
with moreover $\rho_j \in (\cM_j)_\ast$ bounded uniformly in $j$ and such that $\rho = (\rho_j)_\omega$ has support $p_\rho \leq p_\psi$. Set $A = (A_j)_\omega$.
\end{enumerate}




We may view $A$ as in ($\ast$) as an element of $\prod_{j, \omega} L^p(\cM_j, \psi_j)$ because it follows that $\Vert A_j \Vert_p = \Vert A_j^p \Vert_1^{1/p}$ is bounded uniformly in $j$.

On the other hand we have $A^p = (A_j^p)_\omega$ by Lemma \ref{Lem=FunctionaCalculus}.  We may apply   Theorem \ref{Thm=OneMoreSpatial} which shows that $q_\psi(\cH)$ is an invariant subspace for $(A_j^p)_\omega$ on which  $(A_j^p)_\omega$ acts as an unbounded operator, namely  $\frac{d \rho}{ d \psi}$.  This shows that   $q_\psi (A_j)_\omega  q_\psi \in L^p( q_\psi  \cM q_\psi, \psi)$ with
\[
\Vert  q_\psi (A_j)_\omega   q_\psi \Vert_{L^p(\cM, \psi)}^p = \rho(1)  = \lim_{j, \omega} \rho_j(1) = \lim_{j, \omega} \Vert A_j \Vert_{L^p(\cM_j, \psi_j)}^p =
  \Vert (A_j)_\omega \Vert_{ \prod_{j, \omega} L^p(\cM_j, \psi_j) }^p .
\]
So the proof so far shows that $A \mapsto q_\psi A q_\psi$  gives an isometry from the closed linear span in $\prod_{j, \omega} L^p(\cM_j, \psi_j)$ of operators $A$ as in ($\ast$)  to the space $L^p( q_\psi  \cM q_\psi, \psi )$. We need to prove that such operators $A$ are dense in $p_\psi \left( \prod_{j, \omega} L^p(\cM_j, \psi_j) \right) p_\psi$ and that the range equals $L^p( q_\psi  \cM q_\psi, \psi )$.

\vspace{0.3cm}

\noindent {\bf Claim 1.} We claim that the $A$'s from ($\ast$) are densely contained in $p_\psi \left( \prod_{j, \omega} L^p(\cM_j, \psi_j) \right) p_\psi$. 

\noindent {\it Proof of Claim 1.} Firstly let us show that the elements $A$ from ($\ast$) are indeed contained in  $\prod_{j, \omega} L^p(\cM_j, \psi_j)$. To do so it suffices to show that $(A_j^p)_\omega$ is in $p_\psi \prod_{j, \omega} L^1(\cM_j, \psi_j) p_\psi$, which is equivalent to showing that $(\rho_j)_\omega$ is in $p_\psi \left( \prod_{j, \omega} (\cM_j)_\ast \right) p_\psi$ (as the correspondence $\prod_{j, \omega} (\cM_j)_\ast  \simeq \prod_{j, \omega} L^1(\cM_j)$ preserves the left and right module action of $\prod_{j, \omega} \cM_j$ ). But the latter is true by assumption, see ($\ast$).

It remains to show density. By Powers--St\o{}rmer, c.f. Proposition \ref{Prop=PowersStormer}, it suffices to show density of the elements $A^p$ with $A$ as in ($\ast$) in $p_\psi \left( \prod_{j, \omega} L^1(\cM_j, \psi_j)\right) p_\psi$.  Let $\varphi$ be the opposite weight of $\psi$. We have $p_\varphi = p_\psi$ for the supports. The set $\varphi_x^{(-1/2)}$ with $x \in p_\varphi \mathfrak{m}_\varphi^+ p_\varphi$ is dense in $(p_\psi \cM p_\psi)_\ast^+$. So fix such an element $x \in p_\varphi \mathfrak{m}_\varphi^+ p_\varphi$.  Let $\mathcal{A}$ be the set of all series $(x_j)_\omega$ with $x_j \in \mathfrak{m}_{\varphi_j}^+$ such that $\varphi_{j,x_j}^{(-1/2)}$ is bounded uniformly in $j$ and $\varphi_{(x_j)_\omega}^{(-1/2)} \leq \varphi_{x}^{(-1/2)}$. Recall also that $(\varphi_{x_j }^{(-1/2)})_\omega = \varphi_{(x_j)_\omega}^{(-1/2)}$ by Lemma \ref{Lem=SymmetricUltra}.   We have $p_\psi (x_j)_\omega p_\psi \leq p_\psi x p_\psi$ and in fact exactly as in the proof of Lemma \ref{Lem=AnotherApproximation} one can show that $\sup_{a \in \mathcal{A}} p_\psi a p_\psi = p_\psi x p_\psi$. Then also $\sup_{a \in \mathcal{A}} \varphi_{a}^{(-1/2)} = \varphi_x^{(-1/2)}$. This shows that the functionals $\varphi_a^{(-1/2)}$ with $a \in \mathcal{A}$ are dense in $(p_\psi \cM p_\psi)_\ast^+$. Now use that we have an isometric correspondence $(p_\psi \cM p_\psi)_\ast^+ \simeq  p_\psi \left( \prod_{j, \omega} (\cM_j)_\ast^+\right) p_\psi \simeq p_\psi \left( \prod_{j, \omega} L^1(\cM_j, \psi_j)^+\right) p_\psi$ under which the functional $\varphi_{a}^{(-1/2)}, a = (a_j)_\omega \in \mathcal{A}$ corresponds to $A_j^p := \left( \frac{d \varphi_{j, a_j}}{d \psi_j} \right)$. As $A = (A_j)_\omega$ clearly satisfies $(\ast)$ we conclude our claim.

\vspace{0.3cm}

\noindent {\bf Claim 2.} We claim that the $q_\psi A q_\psi$'s from ($\ast$) are dense in $L^p( q_\psi  \cM q_\psi, p_\psi )$.

\noindent {\it Proof of Claim 2.} The proof is essentially the same as in Claim 1.   By the Powers--St\o{}rmer inequality it suffices to show that $(q_\psi A q_\psi)^p = q_\psi A^p q_\psi$ with $A$ as in ($\ast$) is dense in $L^1( q_\psi  \cM q_\psi, \psi )^+$. Then it suffices to show the density in $(q_\psi \cM q_\psi)_\ast^+$ of all positive normal functionals on $q_\psi \cM q_\psi$ that are of the form $(\varphi_{j,x_j}^{(-1/2)} )_\omega$ where $x_j \in \mathfrak{m}_{\varphi_j}^+$ are such that both $x_j$ and $\varphi_{j,x_j}^{(-1/2)}$ are bounded uniformly in $j$. This is exactly what the previous paragraph shows.

\end{proof}

In particular we give a new proof of the following result due to Raynaud, see \cite[Theorem 3.1]{Raynaud}.

\begin{cor}\label{Cor=Lp}
Let $\cM_j$ be von Neumann algebras and let $\cM = \prod_{j, \omega} \cM_j$.
For every $p \in [1, \infty]$ we have an isomorphism,
\[
L^p(   \cM   ) \simeq  \prod_{j, \omega} L^p(\cM_j).
\]
\end{cor}
\begin{proof}
For $p= \infty$ this is by definition so assume that $p \in [1, \infty)$. The proof follows by an inductive limit argument if we could show that the isomorphisms of Theorem \ref{Thm=Lp} agree on intersection taking into account the isomorphism given by a change of weight.
 
 Let $\cN$ be an arbitrary von Neumann algebra. If $\psi_1, \psi_2 \in \cN_\ast^+$ are such that $p_{\psi_1} \leq p_{\psi_2}$ then we have $\Phi_{\psi_1, \psi_2}: L^p(p_{\psi_1} \cN p_{\psi_1}, \psi_1) \subseteq L^p(p_{\psi_2} \cN p_{\psi_2}, \psi_2)$ isometrically and the isomorphism is given by  $\left( \frac{d \rho}{ d\psi_1} \right)^{\frac{1}{p}} \mapsto \left( \frac{d \rho}{ d\psi_2} \right)^{\frac{1}{p}}$ with $\rho \in \cN_\ast$ having support $p_\rho \leq p_{\psi_1}$.

Going back to ultraproducts  take $\psi_1 = (\psi_{1, j})_\omega \in \cM_\ast'$ and $\psi_2 = (\psi_{2,j})_\omega \in \cM_\ast'$ with supports $r_{\psi_1}$ and $r_{\psi_2}$ and opposite supports $p_{\psi_1}$ and $p_{\psi_2}$. We denote the supports of the $\psi_{k,j}$'s  similarly.  Assume that $p_{\psi_1} \leq p_{\psi_2}$. We may and will assume that for every $j$ we have supports $p_{\psi_{1, j}} \leq p_{\psi_{2, j}}$ (indeed one may replace $\psi_{1, j}$ by $r_{\psi_{2,j}} \psi_{1, j} r_{\psi_{2,j}}$ and as $r_{\psi_1} \leq r_{\psi_2} \leq (r_{\psi_{2,j}})_\omega$ we have $(r_{\psi_{2,j}} \psi_{1, j} r_{\psi_{2,j}})_\omega = (r_{\psi_{2,j}})_\omega ( \psi_{1, j})_\omega (r_{\psi_{2,j}})_\omega = ( \psi_{1, j})_\omega$). Now the isomorphisms so far give a commutative diagram (as is clear from their explicit prescriptions),
\[
 \xymatrix{
p_{\psi_1}\left(  \prod_{j, \omega} L^p( p_{\psi_{1,j}} \cM_j p_{\psi_{1,j}}, \psi_{1,j} ) \right) p_{\psi_1} \qquad \ar@{->}[r]^{\:\:\:\: {\rm Thm } \ref{Thm=Lp}}  \ar@{->}[d]^{\prod_{j, \omega} \Phi_{\psi_{1,j}, \psi_{2,j} }}   &   L^p(q_{\psi_1} \cM q_{\psi_1}, \psi_1)  \ar@{->}[d]^{ \Phi_{\psi_1, \psi_2}}   \\
p_{\psi_2} \left( \prod_{j, \omega} L^p( p_{\psi_{2,j}} \cM_j p_{\psi_{2,j}}, \psi_{2,j} ) \right) p_{\psi_2} \qquad \ar@{->}[r]^{\:\:\:\:  \qquad {\rm Thm } \ref{Thm=Lp}}   &   L^p(q_{\psi_2} \cM q_{\psi_2}, \psi_2).
}    
\]
So take an increasing net $\{ \psi_k \}_k$ in $\cM_\ast$ with $p_{\psi_k} \rightarrow 1$. Then taking the inductive limit of the isomorphisms  $p_{\psi_k}  \left( \prod_{j, \omega} L^p(\cM_j, \psi_{k,j}) \right) p_{\psi_k} \simeq L^p( q_{\psi_k}  \cM q_{\psi_k}, \psi_k )$ of Theorem \ref{Thm=Lp} yields the result.
 \end{proof}

\appendix

\section{Spatial derivatives on the standard form}\label{Sect=AppendixA}

We compliment this paper by showing that the spatial derivative attains a nice representation on the standard form, at least if the standard form is represented as
\begin{equation}\label{Eqn=ConcreteStandardForm}
(\cN, L^2(\cN), J, L^2(\cN)^+),
 \end{equation}
 where $J: x \mapsto x^\ast$. See \cite{TerpII} for a discussion on standard forms and $L^p$-spaces.

\vspace{0.3cm}

We fix again notation, keeping the appendix self-contained. Let $\cN$ be a von Neumann algebra. Let $\kappa'$ be some normal, semi-finite, faithful weight on $\cN'$ and set $L^p(\cN) = L^p(\cN, \kappa')$. We assume that $\cN$ is represented in standard form, ie. \eqref{Eqn=ConcreteStandardForm}. For a normal weight $\varphi$ on $\cN$ we write again $D_\varphi$ for $d\varphi/d\kappa'$.
For a normal, semi-finite, faithful weight $\varphi$ on $\cN$ we set
\[
\mathcal{I}_\varphi = \left\{ \omega \in \cN_\ast \mid \textrm{The map } \Lambda_{\varphi}(x) \mapsto \omega(x^\ast), x \in \nphi \textrm{ extends boundedly} \right\}.
 \]
  By the Riesz theorem for $\omega \in \mathcal{I}_\varphi$ there exists a vector $\xi(\omega) \in \cH_\varphi$ such that
  \[
  \langle \xi(\omega), \Lambda_\varphi(x) \rangle = \omega(x^\ast).
   \]
   It follows (see also \cite{CaspersLpf}) that $\xi(_y \varphi) = \Lambda_{\varphi}(y), y \in (\mathcal{T}_{\varphi})^2$ and as $\varphi_y = _{\sigma^\varphi_{-i}(y)} \varphi$ it follows that
\begin{equation}\label{Eqn=XiEmbedding}
\xi(\varphi_y) = \Lambda_\varphi( \sigma^\varphi_{-i}(y) ), \qquad y \in (\mathcal{T}_{\varphi})^2.
\end{equation}
Fix a normal, semi-finite, faithful weight $\psi$ on $\cN'$ and let $\psiop$ be the opposite weight on $\cN$. Recall that it is defined as $\psiop(x^\ast x) = \psi(y^\ast y)$, for $x = J_\psi y^\ast J_\psi$ with $y \in \cN'$ so that $x \in \cN$.   Recall that if $x \in \mathfrak{n}_{\rho}$ then $x D^{\frac{1}{2}}_{\rho} \in L^2(\cN)$ with $\Vert x D^{\frac{1}{2}}_{\rho} \Vert_2 = \rho(x^\ast x)$. We shall also write $D_\psi$ for $D_{\psiop}$ (as $\psi$ is a weight on $\cN'$ there can be no confusion). Note that we have $\npsi = J_\psi \mathfrak{n}_{\psiop}^\ast J_\psi$ and a GNS-construction for $\psi$ is given by  $J z^\ast J \mapsto D_\psi^{\frac{1}{2}} z$, where $z \in \mathfrak{n}_{\psiop}^\ast$.
 For an element $D \in L^1(\cN)^+$ we recall that ${\rm Tr}( D)  = \rho(1)$ where $\rho$ is such that $D = \frac{d\rho}{d\psi}$. We extend this integral/trace to $L^1(\cN)$ by linearity. Now we have the following.

\begin{lem}\label{Lem=Intersection}
Let $\xi \in L^2(\cN)$ and suppose that the mapping $f: \cN_\ast \rightarrow \mathbb{C}: \omega \mapsto \langle \xi(\omega), \xi \rangle$ is bounded. Then there exists $x \in \mathfrak{n}_{\psi^{{\rm op}}}$ such that $\xi = \Lambda_{\psi^{{\rm op}}}(x)$.
\end{lem}
\begin{proof}
Because $f$ is bounded there exists $x^\ast \in \cN$ such that $f(\omega) = \omega(x^\ast)$. We claim that $x \in \npsiop$. Let $y \in (\Tpsiop)^2$ and let $a_j \in \Tpsiop$ be a bounded net converging to 1 in the $\sigma$-strong topology and such that also $ \sigma^{\psiop}_{i/2}(a_j)$ is bounded and convergent to 1  in the $\sigma$-strong topology, c.f. Lemma \ref{Lem=Approx}. Then $x a_j \in \npsiop$ and $x a_j \rightarrow x$ in the $\sigma$-strong topology. Also we have,
\[
\begin{split}
& \langle \xi(\psiop_y), \Lambda_{\psiop}(x a_j) \rangle = \psiop( a_j^\ast x^\ast \sigma^{\psiop}_{-i}(y) )
= \langle \xi( \psiop_{y a_j^\ast} ), \xi \rangle  \\
= & \langle \Lambda_{\psiop}(\sigma^{\psiop}_{-i}(y a_j^\ast) ), \xi \rangle
= \langle J  \sigma^{\psiop}_{i/2}(a_j) J \Lambda_{\psiop}(\sigma^{\psiop}_{-i}(y  ) ), \xi \rangle \\
= & \langle J  \sigma^{\psiop}_{i/2}(a_j) J \xi(\psiop_y), \xi \rangle.
\end{split}
\]
So we see that $\Lambda_{\psiop}(x a_j)$ converges weakly to $\xi$ as the vectors $\xi(\psiop_y)$ are dense in $L^2(\cN)$ (c.f. \eqref{Eqn=XiEmbedding} and \cite[Lemma A.2]{CaspersLpf}). As $\Lambda_{\psiop}$ is $\sigma$-weak/weak closed this implies that $x \in \npsiop$ and moreover that $\Lambda_{\psiop}(x) = \xi$. This proves the claim.
\end{proof}

\begin{lem}\label{Lem=PsiBounded}
The set of  $\psi$-bounded  $D(L^2(\cN), \psi)$ equals $\{ x D_{\psi}^{\frac{1}{2}} \mid x \in \mathfrak{n}_{\psi^{op}} \}$. Furthermore we have $R^\psi(x D_{\psi}^{\frac{1}{2}}) = x$.
\end{lem}
\begin{proof}
Write $\pi_r(z), z \in \cN$ for the right multipliciation mapping $L^2(\cN) \rightarrow L^2(\cN): x \mapsto x z$. In particular $\pi_r(z) \in \cN'$ and $\pi_r(z) = J z^\ast J$. Recall that $\Lambda_\psi: \pi_r(z) \mapsto D_\psi^{\frac{1}{2}} z$ 
 with $z \in \mathfrak{n}_{\psi^{op}}^\ast$ is a GNS-construction for $\psi$. As  for  $x \in \mathfrak{n}_{\psi^{op}}$ we have that $x D_{\psi}^{\frac{1}{2}} \in L^2(\cN)$
   we have for $z \in \mathfrak{n}_{\psi^{op}}^\ast$,
   \[
   R^\psi( x D_\psi^{\frac{1}{2}} ):  D_\psi^{\frac{1}{2}} z = \Lambda_\psi( \pi_r( z) )  \mapsto   \pi_r( z)  x D_\psi^{\frac{1}{2}} =  x D_\psi^{\frac{1}{2}}z.
    \]
    In particular $x D_{\psi}^{\frac{1}{2}}$ is $\psi$-bounded and $R^\psi(x D_\psi^{\frac{1}{2}} ) =  x$.

   Conversely, let $y \in (\Tpsiop)^2$ and take $\xi \in D(L^2(\cN), \psi)  \subseteq L^2(\cN)$ and set $x = R^\psi(\xi)$. Then we have,
   \[
   \begin{split}
   & \overline{\langle \xi(_y \psiop), \xi \rangle} = \overline{ {\rm Tr}(\xi^\ast y D_\psi^{\frac{1}{2}}) }
    =   {\rm Tr}(  D_\psi^{\frac{1}{2}}  y^\ast \xi )= {\rm Tr}(   \sigma^\psi_{-i/2} (y^\ast)  D_\psi^{\frac{1}{2}}\xi )  \\
    = &  {\rm Tr}(  D_\psi^{\frac{1}{2}}\xi   \sigma^\psi_{-i/2} (y^\ast) )
    =      {\rm Tr}(  D_\psi^{\frac{1}{2}}   R^\psi(\xi) \Lambda_{\psi} (\pi_r(  \sigma^\psi_{-i/2} (y^\ast)))  )\\
    =  &  {\rm Tr}( D_\psi^{\frac{1}{2}} x D_\psi^{\frac{1}{2}}  \sigma^\psi_{-i/2} (y^\ast) )
    = {\rm Tr}( x D_\psi y^\ast )
    =    \overline{  {\rm Tr}( y D_\psi x^\ast)}
    =  \overline{  _y \psiop(x^\ast)  }.
   \end{split}
   \]
   This shows that $_y \psiop \mapsto  \langle \xi(_y \psiop), \xi \rangle$ extends boundedly to a map $\cN_\ast \rightarrow \mathbb{C}$ (using   that functionals $_y \psiop$ with $y \in (\Tpsiop)^2$ are dense in $\cN_\ast$). Then applying Lemma \ref{Lem=Intersection} shows that $\xi = \Lambda_{\psiop}(x)$ for some $x \in \npsiop$ and therefore $\xi = x D_\psi^{\frac{1}{2}}$.
\end{proof}

\begin{thm}\label{Thm=Spacial}
Let $\varphi$ be a normal, semi-finite, faithful weight on $\cN$ and let $\psi$ be a normal, semi-finite, faithful weight on $\cN'$. Let ${\rm Dom}(G) = \{ x D_\psi^{\frac{1}{2}} \mid x \in \nphi^\ast \cap \mathfrak{n}_{\psi^{op}} \}$.
Consider the mapping,
\begin{equation}\label{Eqn=SpacialDerivative}
G: {\rm Dom}(G) \rightarrow L^2(\cN): x  D_{\psi}^{\frac{1}{2}} \mapsto D_{\varphi}^{\frac{1}{2}} x.
\end{equation}
Then $G=\left(\frac{d\varphi}{ d\psi}\right)^{\frac{1}{2}}$. In particular $G$ is closed and positive self-adjoint on the domain defined by \eqref{Eqn=SpacialDerivative}.
\end{thm}
\begin{proof}
The domain of $\left(\frac{d\varphi}{ d\psi}\right)^{\frac{1}{2}}$ equals all $\psi$-bounded vectors for which $\langle \varphi, \theta^\psi( x D_\psi^{\frac{1}{2}}, x D_\psi^{\frac{1}{2}}) \rangle$ is finite. So by Lemma \ref{Lem=PsiBounded} these are exactly all vectors of the form $x D_{\psi}^{\frac{1}{2}}$ with $x \in \mathfrak{n}_{\psi^{op}}$ for which $\varphi(x x^\ast)$ is finite, i.e. $x \in \nphi^\ast$. From Lemma \ref{Lem=PsiBounded} we see that for $x,y \in   \mathfrak{n}_{\psi^{op}}$ we have $\theta^\psi( y D_\psi^{\frac{1}{2}}, x D_\psi^{\frac{1}{2}}):   D_\psi^{\frac{1}{2}} z \mapsto  y x^\ast  D_\psi^{\frac{1}{2}} z$. Consequently for $x \in \nphi^\ast \cap \mathfrak{n}_{\psi^{op}}$ we have,
\[
\begin{split}
& \langle  \left( \frac{d \varphi}{d \psi} \right)^{\frac{1}{2}}  ( x D_{\psi}^{\frac{1}{2}}), \left( \frac{d \varphi}{d \psi} \right)^{\frac{1}{2}} (x D_{\psi}^{\frac{1}{2}}) \rangle  \\
 = & \langle \varphi, \theta^\psi( x D_\psi^{\frac{1}{2}}, x D_\psi^{\frac{1}{2}}) \rangle =  \varphi(xx^\ast)
=  \langle  x^\ast D_{\varphi}^{\frac{1}{2}} , x^\ast D_{\varphi}^{\frac{1}{2}} \rangle \\
= & \langle J( x^\ast D_{\varphi}^{\frac{1}{2}}) , J(x^\ast D_{\varphi}^{\frac{1}{2}}) \rangle
=   \langle   D_{\varphi}^{\frac{1}{2}} x,   D_{\varphi}^{\frac{1}{2}} x \rangle
=  \langle  G( x D_{\psi}^{\frac{1}{2}}), G(x D_{\psi}^{\frac{1}{2}}) \rangle.
\end{split}
\]
This implies that $G \subseteq  \left(\frac{d\varphi}{ d\psi}\right)^{\frac{1}{2}}$.
\end{proof}

\begin{cor}
Using the notation of Theorem \ref{Thm=Spacial}, for the Connes cocycle derivative we have:
\[
\left( \frac{D \varphi}{ D \psiop} \right)_t   = D_{\varphi}^{it }  D_{\psi}^{- it }.
\]
\end{cor}
\begin{proof}
For any normal semi-finite faithful weight $\rho$ on $\cN'$ we have (see \cite[Theorem IX.3.8]{TakII})
\[
\left( \frac{D \psiop}{ D \varphi}  \right)_{t} =\left( \frac{d \varphi}{ d\rho}  \right)^{it} \left( \frac{d \psiop}{ d\rho}  \right)^{-it},
\]
so that the proposition follows from Theorem \ref{Thm=Spacial}.
\end{proof}

Finally we need an approximation argument for one of our proofs.

\begin{lem}\label{Lem=CoreOfSpatialDerivative}
Let $\varphi$ be a normal, semi-finite, faithful weight on $\cN$ and let $\psi$ be a normal, semi-finite, faithful weight on $\cN'$. The set $x D_\psi^{\frac{1}{2}}$ with $x \in {\rm span} \: \nphi^\ast \mathfrak{n}_{\psi^{op}}$ is a core for $\left( \frac{d \varphi}{d \psi} \right)^{\frac{1}{2}}$.
\end{lem}
\begin{proof}
Let $\{ e_j \}_j$ be a net in $\nphi^\ast$ converging to $1$ strongly and such that also $\sigma^{\varphi}_{-i/2}(e_j)$ converges to $1$ strongly. Let $x \in \nphi^\ast \cap \mathfrak{n}_{\psiop}$. Then $e_j x \in \nphi^\ast \mathfrak{n}_{\psiop}$.  Then we have that,
\[
e_j x D_{\psi}^{\frac{1}{2}} \rightarrow x D_{\psi}^{\frac{1}{2}},
\]
in norm and
\[
D_{\varphi}^{\frac{1}{2}} e_j x = \sigma_{-i/2}^{\varphi}(e_j) D_{\varphi}^{\frac{1}{2}} x \rightarrow D_{\varphi}^{\frac{1}{2}} x,
\]
in norm. By Theorem \ref{Thm=Spacial} we see that therefore the linear span of $\nphi^\ast \npsi$ is a core for $\left( \frac{d \varphi}{d \psi} \right)^{\frac{1}{2}}$.
\end{proof}

\end{document}